\documentclass{article}

\usepackage[numbers]{natbib} 
\usepackage{color}
\usepackage{amsfonts,amssymb,amsopn,amsmath,mathrsfs,theorem}
\usepackage{hyperref} 
\usepackage{bbm}	  
\usepackage{graphicx} 
\usepackage{verbatim} 
\usepackage[utf8]{inputenc} 
\usepackage{authblk}

\newcounter{subsection1}[section]
\numberwithin{subsection1}{section}

\newtheorem{assump}[subsection1]{Assumption}
\newtheorem{defn}[subsection1]{Definition}
\newtheorem{prop}[subsection1]{Proposition}
\newtheorem{lemma1}[subsection1]{Lemma}
\newtheorem{corollary1}[subsection1]{Corollary}
\newtheorem{remark1}[subsection1]{Remark}
\newtheorem{theorem1}[subsection1]{Theorem}

\newcounter{assumptions}
\newenvironment{proof}{\vspace{1ex}\noindent{\textsc{Proof:}}\hspace{0.5em}}{\hfill\qed\vspace{1ex}}
\newenvironment{ack}{\medskip\noindent\textbf{Acknowledgements}\;}{}

\numberwithin{equation}{section} 
\numberwithin{subsection1}{section}

\begin{document}

\newcommand{\ra}{\Rightarrow} 
\newcommand{\sw}{\subseteq} 
\newcommand{\mc}{\mathcal} 
\newcommand{\mb}{\mathbb} 
\newcommand{\p}{\partial} 
\newcommand{\E}{\mb{E}} 
\newcommand{\Prob}{\mb{P}}
\newcommand{\R}{\mb{R}} 
\newcommand{\C}{\mb{C}} 
\newcommand{\N}{\mb{N}}
\newcommand{\Q}{\mb{Q}}
\newcommand{\Z}{\mb{Z}}
\newcommand{\B}{\mb{B}}
\newcommand{\set}{\,;\,} 
\newcommand{\com}{\leftrightarrow}
\newcommand{\given}{\,|\,} 
\newcommand{\li}{\langle}
\newcommand{\ri}{\rangle}
\newcommand{\wt}{\widetilde}
\newcommand{\1}{\mathbbm{1}}
\newcommand{\cadlag}{c\`{a}dl\`{a}g}

\def\l{\left}
\def\r{\right}
\def\P{\mb{P}}
\def\qed{$\blacksquare$}
\def\vec{\mathbf}

\def\highlight{}

\fontsize{10.4pt}{13pt}
\selectfont

\author[1,2]{Edward Crane\thanks{edward.crane@bristol.ac.uk}}
\author[1,2]{Nic Freeman\thanks{nicfreeman1209@gmail.com}}
\author[2,3]{B\'{a}lint T\'{o}th\thanks{balint.toth@bristol.ac.uk}}

\affil[1]{Heilbronn Institute for Mathematical Research, University of Bristol}
\affil[2]{School of Mathematics, University of Bristol}
\affil[3]{MTA-BME Stochastics Research Group and R\'{e}nyi Institute, Budapest}

\title{Cluster growth in the dynamical Erd\H{o}s-R\'{e}nyi process with forest fires}
\date{\today}

\maketitle

\begin{abstract}

We investigate the growth of clusters within the forest fire model of \citet{RT2009}. The model is a continuous-time Markov process, similar to the dynamical Erd\H{o}s-R\'{e}nyi random graph but with the addition of so-called \textit{fires}. A vertex may catch fire at any moment and, when it does so, causes all edges within its connected cluster to burn, meaning that they instantaneously disappear. Each burned edge may later reappear.

We give a precise description of the process $C_t$ of the size of the cluster of a tagged vertex, in the limit as the number of vertices in the model tends to infinity. We show that $C_t$ is an explosive branching process with a time-inhomogeneous offspring distribution and instantaneous return to $1$ on each explosion. \highlight{Additionally, we show that the characteristic curves used to analyse the Smoluchowski-type coagulation equations associated to the model have a probabilistic interpretation in terms of the process $C_t$.}


\end{abstract}

\allowdisplaybreaks

\section{Introduction}
\label{introsec}

{\highlight
Forest fire models are stochastic interacting particle systems in which the vertices or edges of a graph are gradually switched on, forming growing connected clusters. This growth is counterbalanced by so-called fires; each fire involves the rapid destruction of a single cluster by the switching off of its edges or vertices. Each fire is caused by the random spontaneous ignition of a single vertex, which we will call a lightning strike. The lightning strikes occur independently of the state of the system and are typically taken to be rare events so that on average fires are large. 
}

The evolution of a forest fire model is thus controlled by two competing forces, one that causes clusters to grow slowly and another that causes clusters to burn suddenly. One consequence is that a regime may exist in which the system exhibits self-organized criticality. This means that it is {\highlight driven by its own dynamics} towards a stationary state in which these two opposing forces are precisely balanced. In this state clusters may grow very large before they burn, typified by a heavy-tailed distribution of cluster sizes. Note that the term `self-organized criticality' is a heuristic description of a model's behaviour, rather than any  specific criterion. See \citet{P2012} for a wide ranging discussion of self-organized criticality.

{\highlight
The existence of self-organized criticality in a forest fire model with lightning strikes has been predicted on the lattice $\Z^d$ by \citet{DS1992}. 
Recently, \citet{RT2009} introduced a closely related model, on the complete graph, for which they were able to prove that self-organized criticality occurs in the limit of large system size. It is this model that we study in the present paper; we refer to it as the \textit{Erd\H{o}s-R\'{e}nyi forest fire model}. In both models it is generally accepted that some form of self-organized criticality occurs when the system size tends to infinity and the rate per site at which lightning strikes occur tends slowly to $0$. 
}

The results of \citet{RT2009} are concerned with the limiting behaviour of $v_l(t)$, the fraction of vertices that belong to clusters of size $l\in\N$ at time $t$. Their analysis is based on the important observation that the $v_l(t)$ can be combined into an appropriate generating function $V(t,z)$ which then (in the limit) satisfies a Burgers control problem. {\highlight In this article we paint a further level of detail into the limiting picture; we study the evolution of the size of the cluster of a tagged vertex chosen uniformly at random. We determine the limit of this process as the system size tends to infinity. We show that the limit is an explosive branching process with a time-inhomogeneous offspring distribution and instantaneous return to $1$ on each explosion. Thus in the limit the cluster of our tagged vertex burns at the moment that it becomes infinite.}

We describe the Erd\H{o}s-R\'{e}nyi forest fire model in detail, along with our own results, in Sections \ref{mfffsec} and \ref{cgrowthsec}. We will discuss connections between our own model and other models in the mathematical forest fires literature in Section \ref{litsec}.

\subsection{The Erd\H{o}s-R\'{e}nyi Forest Fire Model}
\label{mfffsec}

We now describe the Erd\H{o}s-R\'{e}nyi forest fire model $(\mc{Z}_t^n)_{t\geq 0}$ introduced in \cite{RT2009}. Let $n\in\N$ and consider $[n]=\{1,2,\ldots,n\}$ as a set of $n$ labelled vertices. At time $t\in[0,\infty)$ the state of the model is described by a multigraph $\mc{Z}^n_t$ with vertex set $[n]$ and unoriented edges; we permit parallel edges and loops. The \textit{cluster} of vertex $k\in [n]$ at time $t$, written $\mc{C}^n_t(k)$, is the connected component of $\mc{Z}^n_t$ containing vertex $k$, i.e. the set of $j\in[n]$ such that there is a path along edges of $\mc{Z}^n_t$ from $k$ to $j$.

Given some (deterministic) initial condition the process $(\mc{Z}^n_t)$ evolves with the following dynamics: 
\begin{itemize}
\item Each unordered pair $(j,k)$ carries a \textit{growth clock} which rings at rate $\frac{1}{n}$. When the growth clock for $(j,k)$ rings we add an edge joining $j$ to $k$ (recall that we permit parallel edges and loops).
\item Each vertex carries a \textit{fire clock} which rings with rate $\lambda_n$ where $\lambda_n\in(0,\infty)$. When this fire clock of vertex $k$ rings, the cluster of $k$ is \textit{burned}: all edges between pairs of vertices in $\mc{C}^n_t(k)$ are instantaneously removed.
\end{itemize}
The growth and fire clocks of distinct edges and vertices are mutually independent. For technical reasons detailed in \cite{RT2009} the process $t\mapsto \mc{Z}^n_t$ is taken to be left-continuous with right limits. 
Consequently it is Markov with respect to the filtration $\mc{F}^n_t=\sigma(\mc{Z}_s\set s\leq t)$.

For each $l=1,\ldots,n$ we define
\begin{equation}\label{vndef}
v^n_l(t)=\frac{1}{n}\big|\{k\in[n]\set|\mc{C}_t(k)|=l\}\big|
\end{equation}
to be the fraction of vertices in $[n]$ that are in a cluster of size $l$ at time $t$. We will think of each vertex as having mass $1/n$, so that the total mass in the system is $1$ and $v_l^n$ is the proportion of mass in clusters of size $l$. 

The effect of the fires results in four different phases of behaviour, as identified in \cite{RT2009}. We restrict our attention to only one (the most interesting) of these phases, where the lightning occurs sufficiently often to prevent the formation of a giant component but also sufficiently rarely that clusters of any fixed finite size are not burned in the limit as $n \to \infty$. The phase in which this occurs is defined by the following assumption.

\begin{assump}\label{crit}
As $n\to\infty$, $\lambda_n\to 0$ and $n\lambda_n\to\infty$.
\end{assump}

Under Assumption \ref{crit}, as $n\to\infty$ a cluster of any constant size $k$ will see a fire at rate $k\lambda_n\to 0$; in other words not at all. However, a cluster which grows to be of size around $\frac{1}{\lambda_n}$ will see lightning at a non-negligible rate. In the process $\mc{Z}^n$, a cluster of size $k\in\N$ and a (distinct) cluster of size $j\in\N$ join together at rate $\frac{kj}{n}$ to form a cluster of size $k+j$. It follows that for each fixed $k$, as $n\to\infty$, we expect $v_k^n(t)$ to see an inflow of mass at rate approximately $\frac{k}{2}\sum_{l=1}^{k-1}v_l^n(t)v_{k-l}^n(t)$ and an outflow at rate approximately $k v_k^n(t)$. The approximations here neglect growth clocks of edges joining vertices within the same cluster of size $k$, and lightning strikes causing clusters of size $k$ to burn, both of which are negligible in the limit as $n \to \infty$ and $\lambda_n \to 0$. In our main result, Theorem \ref{CntoC}, we will exploit these observations, combined with Theorem \ref{TRthmupg} (which improves on the main result of \cite{RT2009} and gives a global description of the behaviour of $\mc{Z}^n$ as $n\to\infty$), to describe the evolution of the size of the cluster of a tagged vertex.

To understand how $v^n_l(t)$ behaves as $n\to\infty$ it is sensible first to examine the simpler case $\lambda_n=0$ (i.e.~no fires) with $v^n_l(0)=\1\{l=1\}$, so that initially we start with only singletons. In this case, $\mc{Z}^n_t$ is simply the Erd\H{o}s-R\'{e}nyi random graph on $[n]$ in which each edge is present, independently of one another, with probability $1-e^{-t/n}$. It is well known that, in the limit as $n\to\infty$, $v^n_l(t)\to v_l(t)$, where $v_l(t)$ is given explicitly in \eqref{EReqssol} below, and the behaviour observed is the following:
\begin{itemize}
\item For $t\in[0,1)$, $l\mapsto v_l(t)$ has an exponential tail and $\sum_{l=1}^\infty v_l(t)=1$.
\item At $t=1$, $l\mapsto v_l(1)$ has a polynomial tail and $\sum_{l=1}^\infty v_l(t)=1$.
\item For $t>1$, $l\mapsto v_l(t)$ has an exponential tail but $\sum_{l=1}^\infty v_l(t)<1$. The reason for this is that a giant component, containing a positive proportion of the vertices, has formed and this component is not picked up by the $v^n_l(\cdot)$ as $n\to\infty$. As $t\to\infty$ this (unique) giant component gradually accumulates all the vertices, so that $\sum_{l=1}^\infty v_l(t)\to 0$.
\end{itemize} 
In fact, as our description above of the cluster growth rates suggests, in this case the limit $t\mapsto(v_l(t))_{l=1}^\infty$ satisfies
\begin{equation}\label{EReqs}
\frac{dv_k(t)}{dt}=\frac{k}{2}\sum\limits_{l=1}^{k-1}v_l(t)v_{k-l}(t)-kv_k(t)
\end{equation}
for all $k\geq 1$. These equations are known as the Smoluchowski coagulation equations with multiplicative kernel. The unique solution to \eqref{EReqs}, with initial condition $v_k(0)=\1\{k=1\}$, is given by
\begin{equation}\label{EReqssol}
v_k(t)=\frac{k^{k-1}}{k!}e^{-kt}t^{k-1}.
\end{equation}

Let us return to Assumption \ref{crit}, \textit{which we assume from now on}. As we said above, this means that in the limit any giant component is killed instantaneously as soon as it appears. However, clusters of any constant size $k\in\N$ do not see fires as $n\to\infty$. As a result, \eqref{EReqs} still holds for $k\geq 2$, but $v_1$ feels an influx of singletons caused by the fires. Such fires can only occur once enough time has passed for the environment to grow clusters of large size; this time is known as the \textit{gelation time} $T_{gel}$. The time $T_{gel}$ depends on the initial condition $v_l(0)=\lim_{n\to\infty}v^n_l(0)$ and (see Section \ref{Tgelsec}) is given by
\begin{equation}
T_{gel}=\left(\sum\limits_{l=1}^\infty lv_l(0)\right)^{-1}\,.\label{E: Tgel}
\end{equation}
Consequently, it is natural to expect that \eqref{EReqs} holds for all $k$ up until $T_{gel}$, whereas after $T_{gel}$ \eqref{EReqs} holds \textit{only} for $k\geq 2$. 

In \cite{RT2009} considerable effort is devoted to showing that (under Assumption \ref{crit}) the limiting process $t\mapsto(v_l(t))_{l=1}^\infty$ satisfies 
$\sum_{l=1}^\infty v_l(t)=1$
for all $t\geq 0$, in contrast to the Erd\H{o}s-R\'{e}nyi case. The result is that this equation replaces the $k=1$ case of \eqref{EReqs}, for all time. To be precise, the system of equations we are interested in as the limit of the $v^n_k(\cdot)$s is described by the following result.

\begin{theorem1}[\citet{RT2009}]\label{TRode}
Suppose that $\sum_1^\infty l^3v_l(0)<\infty$. 
Then there is a unique solution to the following system of equations, called the critical forest fire equations:
\begin{align}
\frac{dv_k(t)}{dt}&= -kv_k(t) + \frac{k}{2}\sum\limits_{l=1}^{k-1}v_l(t)v_{k-l}(t)\hspace{2pc}\text{ for }k\geq 2\label{odegeq2}\\
\sum\limits_{l=1}^\infty v_l(t)&=1.\label{odecrit}
\end{align}
For such a solution, $v_k(t)\in[0,1]$ for all $t\geq 0$ and all $k \in \mathbb{N}$. Further, there exists a function $\varphi:[0,\infty)\to\R$ such that for all $t\neq T_{gel}$ we have
\begin{equation}\label{v1eq}
\frac{dv_1(t)}{dt}=-v_1(t)+\varphi(t).
\end{equation}
The function $v_1$ is continuous on $[0,\infty)$ and continuously differentiable on $(0,T_{gel})\cup(T_{gel},\infty)$. In fact, $\varphi=0$ on $[0,T_{gel})$ and $\varphi$ is both positive and locally Lipschitz on $[T_{gel},\infty)$.

From the same initial conditions, the solution of \eqref{odegeq2}+\eqref{odecrit} coincides with the solution of  \eqref{EReqs} for $t\in[0,T_{gel}]$.  For times $t>T_{gel}$ the solutions do not coincide. For $t \ge T_{gel}$ the solution of \eqref{odegeq2}+\eqref{odecrit} satisfies $\sum_{l=k}^\infty v_l(t)\sim\sqrt{\frac{2 \varphi(t)}{\pi} } k^{-1/2}$ as $k \to \infty$.
\end{theorem1}

\begin{remark1}
Note that the functions $(v_l)$ do not depend on $(\lambda_n)$, except through Assumption \ref{crit}.  {\highlight
  For each fixed $n$ the random functions $v_l^n$ do depend on $\lambda_n$, and finer analysis would be needed to see this dependence in the limit.  
}
\end{remark1}

\begin{remark1}
The function $\varphi$ is the limiting rate at which mass within $\mc{Z}^n$ burns as $n\to\infty$, where each vertex is thought of as having mass $1/n$. Note that $\varphi$ is not continuous at $T_{gel}$.
\end{remark1}

Note that the Smoluchowski equations \eqref{EReqs} can be solved separately for $k=1,2,\ldots$ in turn; consequently existence and uniqueness of solutions is not difficult to prove. However, the critical forest fire equations \eqref{odegeq2}+\eqref{odecrit} form a genuinely infinite system that is significantly harder to work with. As was observed in \cite{RT2009}, equations \eqref{odegeq2}+\eqref{odecrit} can be recast (using a suitable moment generating function) as a Burgers control problem (see equation \eqref{E: Burgers}), where $\varphi$ is the control function.

From now on we take $v_l(t)$ as given by Theorem \ref{TRode}. As part of Theorem 2 of \cite{RT2009} it is shown that for each $\epsilon>0$ and each $t\geq 0$,
\begin{equation}\label{fixedtimeTRthm}
\Prob\left[|v^n_l(t)-v_l(t)|>\epsilon\right]\to 0
\end{equation}
as $n\to\infty$, providing that $v_l^n(0)\to v_l(0)$ and $\sum l^3 v_l(0)<\infty$. In fact, convergence was proven in a slightly stronger sense than \eqref{fixedtimeTRthm} and we will show that convergence holds in a stronger sense still; in Section \ref{FFFtopsec} we state the convergence theorem of \cite{RT2009} precisely and show that it can be upgraded into locally uniform convergence in probability, leading to the following result.

\begin{theorem1}\label{TRthmupg}
Suppose that, for each $l\in\N$, $v^n_l(0)\to v_l(0)$ as $n\to\infty$, where $\sum_l l^3 v_l(0)<\infty$. Then for each $\epsilon>0$, and each $T>0$,
$$\Prob\left[\sup\limits_{l\in\N}\sup\limits_{s\in[0,T]}\left|v^n_l(s)-v_l(s)\right|>\epsilon\right]\to 0$$
as $n\to\infty$.
\end{theorem1}

Recall that, for each $l$, the function $t\mapsto v_l(t)$ is continuous. So far, in keeping with \cite{RT2009}, we have used left-continuous $v^n_l(t)$ (and left-continuous $\mc{Z}^n$). Note that Theorem \ref{TRthmupg} would also hold if we replaced $v^n_l$ by its {\cadlag} version.

{\highlight
In fact, with Theorem \ref{TRthmupg} in hand it is advantageous to switch from working with left-continuous paths to working with {\cadlag} paths (i.e.~right-continuous with left limits). Having {\cadlag} paths will be helpful to us because our main result (Theorem \ref{CntoC}) is a result about convergence of jump processes and as part of its proof we will use standard results concerning martingales and stopping times.

To avoid unnecessary notation we will use the same symbols to refer to both versions; our convention is that up to this point and for the duration of Section \ref{FFFtopsec} (in which Theorem \ref{TRthmupg} is proved) we use left-continuous paths but in all other sections (and for the remainder of Section \ref{introsec}) we use {\cadlag} paths. 
}

\subsection{Cluster Growth}
\label{cgrowthsec}

The sequence $v^n_l(\cdot)$ characterizes the globally averaged behaviour of (the size of) all clusters present in $\mc{Z}^n_\cdot$ as $n\to\infty$. Our aim in this paper is to paint a further level of detail into this picture by describing the behaviour of the cluster associated to a vertex chosen uniformly at random within $\mc{Z}^n$. 

Let $p$ be a vertex sampled uniformly at random from $[n]$ (independently of $\mc{Z}^n$) and set
$$C^n_t=|\mc{C}^n_t(p)|.$$
In order to understand the behaviour of $C^n$, let us consider heuristically the evolution when $C^n_t=k$. 
{\highlight 
In this case the total rate of the growth clocks of edges with at least one endpoint in $\mc{C}^n_t(p)$ is $k(1+\mc{O}(\frac{k}{n}))$ as $n \to \infty$, uniformly in $k$ (see \eqref{Rnk} for the exact rate).
}
 As $n\to\infty$ we typically have $n\gg k$ so, when the next new edge is connected to $\mc{C}^n_t(p)$, it is very unlikely for both the endpoints of this edge to be within $\mc{C}^n_t(p)$. Consequently the corresponding cluster $\mc{C}'$ to which $\mc{C}^n_t(p)$ connects looks very similar to a size biased sample of the clusters in $\mc{Z}^n_t$, that is $\Prob\left[|\mc{C}'|=j\right]\approx \E[v^n_j(t)]$. 

In this paper we define and study $C$, a certain Markov branching process in a varying environment.  We will show that $C$ is the limit of the processes $C^n$ as $n \to \infty$.
In view of Theorem \ref{TRthmupg}, if $t>T_{gel}$ and $C_t=k$, we expect $C_t$ to increase at rate $k$ to a size $k+L$ where $L$ is a random variable whose distribution satisfies $\Prob[L \ge l] \asymp l^{-1/2}$. Such a process is explosive in finite time.

When $C^n_t$ has size $k$ it sees a fire at rate $k\lambda_n$, which tends to zero as $n\to\infty$. However, if $C^n_t$ manages to grow large enough (in particular, to size $C^n_t\gg \frac{1}{\lambda_n}$) then the cluster $\mc{C}^n_t(p)$ will burn and $C^n_t$ will return to $1$. It is not obvious that $C^n_t$, started at size $k=\mc{O}(1)$, can grow to size $\frac{1}{\lambda_n}$ in $\mc{O}(1)$ time but in Section \ref{couplingsec} we will show that in fact this does occur. Consequently, in the limit as $n\to\infty$ we expect to see an instantaneous return to $1$ at each explosion time.

Let $E=\N$ and equip $E$ with the topology such that $\lim_{n\to\infty}n=1$ and $1$ is the only non-isolated point of $E$. Note that $E$ is compact and that the topology on $E$ is metrizable, for example by the metric $d_E(i,j) = |f(i) - f(j)|$, where $f(i) = 1/i$ for $i \ge 2$ and $f(1) = 0$. We will use $E$ as the state space for $C$, so that $C$ is continuous at each of its explosion times. 

We are now in a position to state our main result.

\begin{defn}\label{Cdef}
Let $t \mapsto C_t$ be the unique {\cadlag} $E$-valued strongly Markov process such that:
\begin{itemize}
\item The distribution of $C_0$ is $k\mapsto v_k(0)$.
\item $C$ jumps out of state $k$ with rate $k$. When such a jump occurs at (the random) time $\tau$ then, conditionally on $\tau$, the value of $C$ increases by $L$, sampled according to the distribution $\Prob_\tau\left[L=l\right]=v_l(\tau)$.
\end{itemize}
\end{defn}

\begin{theorem1}\label{CntoC} Suppose $\sum l^3v_l(0)<\infty$ and that $\lim_{n\to\infty}v^n_l(0)= v_l(0)$ for each $l$. Then there exists a coupling of $C^n$ and $C$ such that, for each $\epsilon>0$ and $T>0$,
$$\Prob\left[\sup\limits_{s\in[0,T]}d_E(C^n_s,C_s)>\epsilon\right]\to 0$$
as $n\to\infty$.
\end{theorem1}

\begin{remark1}
The coupling mentioned in Theorem \ref{CntoC} is constructed explicitly as part of the proof. 
\end{remark1}

Note that Definition \ref{Cdef} provides a clear description of how the increments of $C$ behave, but it does not offer a characterization of the distribution at fixed time. We rectify this with the following result, which will be proved as part of argument leading to Theorem \ref{CntoC}.

\begin{prop}\label{Ctdist}
Suppose that $\sum_l l^3v_l(0)<\infty$. Then, for all $t\in[0,\infty)$ and all $l\in\N$, $\Prob\left[C_t=l\right]=v_l(t)$.
\end{prop}

Recall that the growth of the cluster of any fixed vertex in $\mc{Z}^n_t$ is driven by sampling increments from the (random) cluster size distribution of $\mc{Z}^n_t$, with a small modification to correct for the possibility that a new edge forms a cycle. In the limit as $n \to \infty$ the cluster size distribution becomes deterministic, so we expect the local limit of the cluster size of a fixed vertex to be strongly Markov (with respect to its generated filtration), even though $C^n_t$ is not. In the finite model $\mc{Z}^n$, exchangeability implies that the distribution of the size of the cluster of a randomly sampled point is equal to the size biased distribution of the global distribution of cluster sizes. Proposition \ref{Ctdist} shows that this property passes meaningfully through the limit.
The heuristic that we have just given for why Proposition \ref{Ctdist} should hold true relies on Theorem \ref{CntoC}, whereas in fact Proposition \ref{Ctdist} will be a key step in our proof of Theorem \ref{CntoC}.

{\highlight{
\subsection{Structure of the paper}

In Section~\ref{litsec} we place the Erd\H{o}s-R\'{e}nyi forest fire model and our results in the context of some related models in the literature on coagulation-fragmentation processes.

In Section~\ref{FFFtopsec} we prove Theorem~\ref{TRthmupg}. This section refers to technical details of \cite{RT2009}. The main object of Section~\ref{Ctdistsec} is the proof of Proposition~\ref{Ctdist}. This is done by analysing a linear control problem which characterizes the distribution of the process $C_t$. In Lemma~\ref{L: explosion prob} we provide a probabilistic interpretation of the associated characteristic curves that may be of independent interest. In Section \ref{explodesiosec} we establish the long-term average behaviour of $\varphi$ and, as a consequence, we show that $C_t$ explodes infinitely often. Finally, in Section~\ref{couplingsec} we prove Theorem \ref{CntoC} by constructing a coupling between the (finite) Erd\H{o}s-R\'{e}nyi forest fire model $\mathcal{Z}^n$ and the process $C_t$. An outline of this coupling is given at the start of Section \ref{couplingsec}. 

Sections \ref{FFFtopsec}, \ref{Ctdistsec} and \ref{couplingsec} can be read essentially independently of one another. Section \ref{Ctdistsec} does not rely on anything from Section \ref{FFFtopsec}, whilst Section \ref{couplingsec} relies only on Sections \ref{FFFtopsec} and \ref{Ctdistsec} through the statements of Theorem \ref{TRthmupg} and Proposition \ref{Ctdist}.
}}

Throughout Sections~\ref{FFFtopsec},~\ref{Ctdistsec} and~\ref{couplingsec}, as well as Assumption \ref{crit} we assume without further comment the hypotheses on the initial conditions that appear in the statements of our main results, namely $\sum_l l^3 v_l(0)<\infty$ and that $v^n_l(0)\to v_l(0)$ as $n\to\infty$ for each $l\in\N$.

\subsection{Relationships to other models}
\label{litsec}

In general, long range interactions between large clusters are not easy to analyse rigorously, or even simulate. As a consequence, rigorous results concerning forest fire models are not common. One model in particular deserves special mention in comparison to our own. The Drossel-Schwabl model (introduced in \cite{DS1992}) differs from our own model in two important respects: its underlying graph is the lattice $[-n,n]^d$ and growth clocks correspond to vertices rather than edges. Despite receiving much attention in the physics literature, in the appropriate limit of the stationary two dimensional Drossel-Schwabl model, it is not even known whether the probability that the origin is occupied is less than or equal to $1$ (as was noted by van den Berg and Brouwer \cite{BB2006}, who investigate a closely related question).

\citet{SDS2002} gave a detailed non-rigorous description of the two dimensional Drossel-Schwabl model in its stationary state. They showed that in this case self-organized criticality occurs through the appearance of two qualitatively different types of fires, occurring simultaneously within the model but on different scales. Such multi-scale behaviour is often associated to self-organized criticality; see \citet{P2012} for a detailed survey of the physics literature. 

There is a natural connection between forest fire models and percolation, resting on the heuristic observation that taking a forest fire model and suppressing its fires results in a percolation model. As we saw in Section \ref{mfffsec}, in the Erd\H{o}s-R\'{e}nyi forest fire model this connection leads to the dynamical Erd\H{o}s-R\'{e}nyi model. 

There has been recent interest in building a forest fire mechanism into percolation on $\Z^d$, by starting with supercritical percolation, burning the infinite cluster (but keeping the other finite clusters) and then asking what additional edge density must be added in order to create a new infinite cluster. This question was posed by van den Berg and Brouwer in \cite{BB2004} and was investigated for $d\geq 7$ by \citet{ADK2013} and for $d=2$ by \citet{KMS2013}. 

The frozen percolation model, introduced by \citet{A1987}, is another hybrid of forest fires and percolation. In frozen percolation vertices in clusters that become infinite are instantly removed from the model and never return. Consequently, the total number of vertices in the model decreases as time passes; unlike \eqref{v1eq} there is no influx of mass back into $v_1(t)$ and this makes the model somewhat easier to analyse. Frozen percolation is known to exhibit self-organized criticality and limit theorems concerning the size of the cluster of a typical vertex in mean field frozen percolation have been established in \citet{R2009}. A forthcoming work of Martin and R\'{a}th will give a precise description, in terms of the multiplicative coalescent, of the behaviour of the largest clusters in mean field frozen percolation with $\lambda_n=n^{-1/3}$.

Returning to forest fire models, attempts have been made to construct limits in the form of infinite interacting particle systems. In the case where the underlying graph is the integer lattice $[-n,n]^d$ and $\lambda_n=\lambda\in(0,\infty)$ stays constant as $n\to\infty$, it was shown in \citet{D2006, D2006a, D2009a} that such a a limiting process exists. 
\citet{S2012} showed that this process has a stationary distribution. 
Note that in this limit clusters will always burn while they are still of $\mc{O}(1)$ size. 

 {\highlight
In the case where the underlying graph is a regular tree, and with $\lambda_n$ tending slowly, but not too slowly, to $0$ (or rather, along a suitable subsequence where $n$ is the number of vertices of a regular tree), \citet{G2014} has shown that the limit, up to and including the gelation time, is a dynamical version of self destructive percolation. \citet{G2014a} considers the case where the underlying graph is the upper half plane of $\Z^2$ and, with a slightly different approximation scheme and burning mechanism, establishes tightness (but not uniqueness) of the limit.
}

One dimensional forest fire models have received much more rigorous treatment than dimensions greater than one; like our own model they have simplified spatial structure. Notably, \citet{BF2010} constructed a particle system limit of one dimensional forest fire models, in the appropriate scaling regime where $\lambda$ tends to zero. {\highlight{In \cite{BF2012} the same authors find interesting limits of the equilibria of an infinite system of coupled ODEs, which was obtained from a one dimensional forest fire model by a mean field approximation. These equations are similar to the critical forest fire equations discussed in the present paper, but have a constant coalescence kernel instead of a multiplicative one, which is to say that large clusters wait as long to coalesce as small ones do, although clusters burn at a rate proportional to their size.}}  \citet{B2011} investigates a forest fire version of Knuth's parking model, which is related to hashing with linear probing and, in a similar vein, \citet{BT2001} investigate a forest fire model related to signal processing and show that it exhibits self organized criticality. We refer the reader interested in the one dimensional case to the references therein.

We have already introduced the relationship between our model and the Smoluchowski coagulation equations with multiplicative kernel, in \eqref{EReqs} and Theorem \ref{TRode}. A wide ranging survey of Smoluchowski coagulation equations and associated stochastic systems and be found in \citet{A1999}. A derivation of Smoluchowski's equation as the limit in law of an appropriate (stochastic) particle system, along with existence and uniqueness results corresponding to quite general kernels can be found in \cite{N1999, N2000}.

\citet{DFT2002} study what, in our terminology, is the growth process of the cluster of a tagged particle in the environment associated to Smoluchowski coagulation equations with more general kernels, up to time $T_{gel}$. In particular, they use an analogue of Proposition \ref{Ctdist} to construct solutions to Smoluchowski coagulation equations over time $[0,T_{gel})$. By contrast, for \eqref{odegeq2}+\eqref{odecrit} the conservation of mass beyond $T_{gel}$ means that the tagged particle exhibits interesting behaviour after $T_{gel}$ but existence and uniqueness of solutions is already known (see Theorem \ref{TRode}).

\section{The Space Of Forest Fire Evolutions}
\label{FFFtopsec}

Theorem 2 of \citet{RT2009}, which we seek to improve upon in this section, identifies the limit of the process $t\mapsto (v_l^n(t))_{l=1}^n$. In order to understand their result we must first describe the space in which $(v_l^n(t))_{l=1}^n$ lies. 

Let $T>0$. Let $\mathscr{W}_T$ be the space of paths $w:[0,T]\to[0,1]$ that are left-continuous with right limits and are of bounded variation. Note that each such path $w(\cdot)$ can be written as 
\begin{equation}\label{muw}
w(t)=\mu_w[0,t)
\end{equation}
for some finite signed measure $\mu_w$ on $[0,T]$. For $w^{n},w\in\mathscr{W}_T$ we say that
$w^{n}\to w$ if and only if $\mu_{w^{n}}\to \mu_{w}$ weakly as $n\to\infty$. It is shown in \cite{RT2009} that there is a topology inducing this convergence under which $\mathscr{W}_T$ is a Polish space. 

Let
$\mathscr{V}=\left\{\vec{u}=(u_l)_{l=1}^\infty\set u_l\geq 0\text{ and }\sum_{l=1}^\infty u_l\leq 1\right\}.$
and for each $T>0$ let 
$$
\mathscr{E}_T=\left\{\vec{u}:[0,T]\to\mathscr{V}\set\text{ for each }l, u_l(\cdot)\text{ is left-continuous and of bounded variation}\right\}
$$
If $\vec{u}^{n}=(u^{n}_l(\cdot))\in\mathscr{E}_T$ and $\vec{u}=(u_l(\cdot))\in\mathscr{E}_t$ then we say 
\begin{equation}\label{Etop}
\vec{u}^{n}\to\vec{u}\hspace{0.5pc}\iff\hspace{0.5pc}\text{for each }l,\; u^{n}_l\to u_l\text{ in }\mathscr{W}_T
\end{equation}
where, again, the convergence on both sides is as $n\to\infty$. This topology makes $\mathscr{E}_T$ Polish. The space $\mathscr{E}_T$ is referred to in \cite{RT2009} as the space of `forest fire evolutions' over the time interval $[0,T]$.

We set $\vec{v}^n=(v^n_l(\cdot))$ and $\vec{v}=(v_l(\cdot))$. We consider these as elements of $\mathscr{E}_T$ (for each $T$) without comment by restricting the domains of the paths $v^n_l$ and $v_l$ to the time interval $[0,T]$. 

\begin{theorem1}[Theorem 2, \cite{RT2009}]\label{TRthm}
Suppose that $\sum l^3v_l(0)<\infty$ and that $v^n_l(0)\to v_l(0)$ for each $l\in\N$ as $n\to\infty$. Then, for each $T>0$, $\vec{v}^n\to\vec{v}$ in probability in $\mathscr{E}_T$.
\end{theorem1}

We will now upgrade Theorem \ref{TRthm} into Theorem \ref{TRthmupg}. In order to do this we will need to look a little way inside of the proof of Theorem \ref{TRthm} but first we record an elementary result.

\begin{lemma1}\label{elemlemma}
Let $T>0$ and let  $w\in\mathscr{W}_T$ be continuous. For each $n\in\N$ let $w^n$ be a $\mathscr{W}_T$ valued random variable such that the path $w^n$ is increasing and suppose that $w^n\to w$ in probability in $\mathscr{W}_T$. Then, for each $\epsilon>0$, $\Prob[\sup_{s\in[0,T]}|w^n(s)-w(s)|>\epsilon]\to 0$ as $n\to\infty$.
\end{lemma1}


\begin{proof}
We have $w^n\to w$ in probability and, since $\mathscr{E}_T$ is separable, we may apply the Skorohod Representation Theorem and assume that $w^n\to w$ almost surely (after a change of our underlying probability space). Thus $\mu_{w^n}\to \mu_w$ almost surely, in the sense of the weak topology on measures on $[0,T]$.

Since $w$ is continuous the signed measure $\mu_w$ defined by \eqref{muw} is non-atomic and hence $w(t)=\mu_w[0,t]$. {\highlight Since $\mu_{w^n}$ is a non-negative measure for all $n$, $\mu_w$ is also a non-negative measure.} Further, $\mu_w\{0,t\}=0$ so $[0,t]$ is a $\mu_w$-continuity set. The Portmanteau Theorem thus implies that $\mu_{w^n}[0,t]\to \mu_{w}[0,t]$ almost surely. Similarly, $\mu_{w^n}\{t\}\to \mu_w\{t\}=0$ almost surely so we can conclude that
\begin{equation}\label{elunifp}
w^n(t)=\mu_{w^n}[0,t]-\mu_{w^n}\{t\}\to \mu_{w^n}[0,t]=w(t)
\end{equation}
almost surely.

Now let $\epsilon>0$. Since $[0,T]$ is compact $w$ is uniformly continuous and hence there exists $\delta>0$ such that for all $|s-t|<\delta$, $|w(t)-w(s)|<\epsilon$. Let $M$ be such that $(M-1)\delta <T\leq M\delta$. Note that the set $\mc{T}=\{0,\delta,2\delta,\ldots,(M-1)\delta,T\}$ is finite, hence
\begin{equation}\label{eldivs}
\Prob\left[\exists t\in\mc{T}, |w^n(t)-w(t)|>\epsilon\right]\to 0\;\text{ as }n\to\infty.
\end{equation}

Fix $s\in[0,T]$. Then there is some $k$ such that $(k-1)\delta<s<(k+1)\delta$. Since $w^n$ is increasing we have
$w^n((k-1)\delta)\leq w^n(s)\leq w^n((k+1)\delta)$
and from the uniform continuity we have
\begin{align*}
|w(s)-w((k-1)\delta)|\leq 2\epsilon,\hspace{1pc}
|w(s)-w((k+1)\delta)|\leq 2\epsilon.
\end{align*}
On the complement of the event in \eqref{eldivs} we have also that 
\begin{align*}
|w^n((k-1)\delta)-w((k-1)\delta)|\leq \epsilon,\hspace{1pc}
|w^n((k+1)\delta)-w((k+1)\delta)|\leq \epsilon
\end{align*}
and we are thus able to conclude that
$\Prob[\sup_{s\in[0,T]}|w^n(s)-w(s)|\geq 3\epsilon]\to 0$
as $n\to\infty$, which completes the proof.
\end{proof}

We now describe the evolution of $v^n_k$ in terms of the two forces affecting it: coagulating clusters and burning clusters. Define $Q^n_{j,k}(t)$ to be the number of times during $[0,t]$ that a cluster of size $k$ and a cluster of size $j$ coagulate to form a cluster of size $j+k$, within $\mc{Z}^n$. For $j>1$ define $R^n_{j}$ to be the number of times during $[0,t]$ that a cluster of size $j$ burns. Then set
\begin{align*}
q^n_{j,k}(t)&=\frac{Q^n_{j,k}(t)}{n},\hspace{1pc} 
q^n_k(t)=\sum\limits_{l=1}^\infty q_{k,l}(t),\hspace{1pc}
r^n_j(t)=\frac{R^n_{j}(t)}{n},\hspace{1pc}
r^n(t)=\sum\limits_{k=2}^\infty r^n_k(t).
\end{align*}
It is readily seen from the definition of $\mc{Z}^n$ that
\begin{equation}\label{vqrn}
v^n_k(t)=v^n_k(0)+\frac{k}{2}\sum\limits_{l=1}^{k-1}q_{l,k-l}^n(t)-kq^n_k(t)-r^n_k(t)+\1\{k=1\}r^n(t).
\end{equation}
We now collate together results from Proposition 1, equations (19), (37) and Theorem 2 of \cite{RT2009}. There are continuous functions $q_{j,k}(\cdot)$ and $r_k(\cdot)$ such that for all $j,k$
\begin{equation}\label{qrconv}
q_{j,k}^n\to q_{j,k},\hspace{1pc} q_k^n\to q_k,\hspace{1pc} r_k^n\to r_k,\hspace{1pc} r^n\to r
\end{equation}
in $\mathscr{W}_T$ for any $T>0$ and, further,
$$v_k(t)=v_k(0)+\frac{k}{2}\sum\limits_{l=1}^{k-1}q_{l,k-l}(t)-kq_k(t)-r_k(t)+\1\{k=1\}r(t)$$
where $q_k(t)=\sum_{l=1}^\infty q_{k,l}(t)$ and $r(t)=\sum_{k=2}^\infty r_k(t)$. In fact $r_k(\cdot)=0$ (heuristically, this is because a cluster of size $k$ burns at rate $k\lambda_n\to 0$) but we will continue to write $r_k$ for symmetry.

\begin{proof}[Proof of Theorem \ref{TRthmupg}]
Let $\epsilon>0$ and $T\in(0,\infty)$. We will prove the theorem in two steps, the first of which is to show that for each $k\in\N$,
\begin{equation}\label{TRthmupgpre}
\Prob\left[\sup\limits_{s\in[0,T]}\left|v^n_k(t)-v_k(t)\right|>\epsilon\right]\to 0.
\end{equation}
as $n\to\infty$.

Let us first look at $k\geq 2$. In this case we can write
\begin{align*}
v^n_k(t)-v_k(t)&=v^n_k(0)-v_k(0)+\overbrace{\left(\frac{k}{2}\sum\limits_{l=1}^{k-1}q^n_{l,k-l}(t)+kq_k(t)+r_k(t)\right)}^{f_1(t)}
-\overbrace{\left(\frac{k}{2}\sum\limits_{l=1}^{k-1}q_{l,k-l}(t)+kq_k^n(t)+r^n_k(t)\right)}^{f_2(t)}
\end{align*}
We have that $v^n_k(0)-v_k(0)$ converges to zero as $n\to\infty$. Moreover, both $f_1$ and $f_2$ are increasing functions and elements of $\mathscr{W}_T$. Equation \eqref{qrconv} implies that $f_1$ and $f_2$ both converge (in $\mathscr{W}_T$) to $$\frac{k}{2}\sum_{l=1}^{k-1}q_{l,k-l}(t)+kq_k(t)+r_k(t).$$ Applying Lemma \ref{elemlemma} to $f_1$ and $f_2$ respectively, we obtain that in fact $v_k^n-v_k$ tends to $0$ locally uniformly in probability, which proves \eqref{TRthmupgpre} for $k\geq 2$.

The case $k=1$ remains. In this case, $v^n_1(t)-v_1(t)$ has the additional term 
$r^n(t)-r(t).$
As we have mentioned above, $r_k=0$ so $r(t)=\sum_{k=2}^\infty r_k(t)=0$. From \eqref{qrconv} we have $r^n\to r$ in $\mathscr{W}_T$ and by Lemma \ref{elemlemma} we have that $r^n\to r$ locally uniformly in probability. Combining this fact with the argument used in the $k\geq 2$ case, we have proved the $k=1$ case of \eqref{TRthmupgpre} and thus completed the proof of \eqref{TRthmupgpre}.

We now deduce Theorem \ref{TRthmupg} from \eqref{TRthmupgpre}. By Theorem \ref{TRode}, for each $k$, $t\mapsto v_k(t)$ is continuous on $[0,T]$. By Dini's theorem, we can choose $K\in\N$ be such that 
\begin{equation}\label{>Kbound}
\sup_{s\in[0,T]}\sum_{k=K+1}^\infty v_k(s)\leq \frac{\epsilon}{3}.
\end{equation}
Hence also $\sup_{s\in[0,T]}\sup_{k>K}v_k(s)\leq \frac{\epsilon}{3}$.
Using \eqref{TRthmupgpre}, let $N\in\N$ be such that for all $n\geq N$,
\begin{equation}\label{TRupg1}
\Prob\left[\exists k\leq K, \sup\limits_{s\in[0,T]}\left|v^n_k(t)-v_k(t)\right|\geq\frac{\epsilon}{3K}\right]\leq \epsilon.
\end{equation}
Using \eqref{vndef} and \eqref{odecrit}, we note that, for $k>K$,
\begin{align}
\sup\limits_{s\in[0,T]} v^n_k(s)
&\leq \sup\limits_{s\in[0,T]}\sum\limits_{l=K+1}^\infty v^n_l(s)\notag\\
&=\sup\limits_{s\in[0,T]}\left(1-\sum\limits_{l=1}^Kv_l(s)-\sum\limits_{k=1}^K(v^n_l(s)-v_l(s))\right)\notag\\
&=\sup\limits_{s\in[0,T]}\left(\sum\limits_{l=K+1}^\infty v_l(s)-\sum\limits_{l=1}^K(v^n_l(s)-v_l(s))\right)\notag\\
&\leq\frac{\epsilon}{3}+\sum\limits_{l=1}^K\sup\limits_{s\in[0,T]}|v^n_l(s)-v_l(s)|\label{TRupg2}
\end{align}
Note that to obtain the last line of the above we used \eqref{>Kbound}. Note also that the final line is independent of $k$. Using\eqref{TRupg2}, followed by another application of \eqref{>Kbound} and then two applications of \eqref{TRupg1}, we have
\begin{align*}
&\Prob\left[\sup\limits_{k\in\N}\sup\limits_{s\in[0,T]}\left|v^n_k(s)-v_k(s)\right|>\epsilon\right]\\
&\hspace{3pc}\leq \Prob\left[\sup\limits_{k\leq K}\sup\limits_{s\in[0,T]}\left|v^n_k(s)-v_k(s)\right|>\epsilon\right]+\Prob\left[\sup\limits_{k>K}\sup\limits_{s\in[0,T]}v^n_k(s)+v_k(s)>\epsilon\right]\\
&\hspace{3pc}\leq \Prob\left[\sum\limits_{k=1}^K\sup\limits_{s\in[0,T]}|v^n_k(s)-v_k(s)|>\frac{\epsilon}{3}\right] 
+ \Prob\left[\frac{\epsilon}{3}+\sum\limits_{k=1}^K\sup\limits_{s\in[0,T]}|v^n_k(s)-v_k(s)|+\sup_{k> K}\sup_{s\in[0,T]}v_k(s)\geq \epsilon\right]\\
&\hspace{3pc}\leq \Prob\left[\sum\limits_{k=1}^K\sup\limits_{s\in[0,T]}|v^n_k(s)-v_k(s)|>\frac{\epsilon}{3}\right] 
+ \P\l[\sum\limits_{k=1}^K\sup\limits_{s\in[0,T]}|v^n_k(s)-v_k(s)|\geq\frac{\epsilon}{3}\r]\\
&\hspace{3pc}\leq 2\epsilon.
\end{align*}
This completes the proof of Theorem \ref{TRthmupg}.
\end{proof}

\section{Cluster Growth in the Limiting Process}
\label{Ctdistsec}

In this section we investigate the process $C$ which, in the next section, will be shown to be the limit of $(C^n)$. {\highlight Recall that, from this point on, we use {\cadlag} versions of all processes.} The main goal of this section is to prove Proposition \ref{Ctdist}, which states that $\mathbb{P}[C_t = \ell] = v_\ell(t)$ for all $t > 0$. Our strategy for the proof is as follows. Recall that in Theorem \ref{TRode} we gave a system of ODEs for the evolution of the $(v_k)$; naturally they can also be expressed as a system of integral equations. We define
$$u_k(t)=\Prob\left[C_t=k\right]$$
and try set up a similar system of integral equations for the $(u_k)$. We then have a description of the evolution of $u_k-v_k$ and seek to show that in fact $u_k-v_k$ is identically zero. 

It is convenient to use the probability generating functions
\begin{align*}
X_t(z)&=\sum\limits_{k=1}^\infty z^k v_k(t),\hspace{1pc}
Y_t(z)=\sum\limits_{k=1}^\infty z^k u_k(t)=\mathbb{E}\left(z^{C_t}\right).
\end{align*}
{\highlight 
Both $X_t$ and $Y_t$ are power series in $z$ with non-negative coefficients summing to $1$, which therefore converge uniformly on the closed unit disc in the complex plane. Thus they define analytic functions on the open unit disc $\mathbb{D}$ with continuous extension to the closed unit disc $\overline{\mathbb{D}}$, and we have $X_t(1)=Y_t(1)=1$. We will mostly be concerned with the behaviour of $X_t(z)$ and $Y_t(z)$ for $z \in [0,1]$. 
}

We define
$$Z_t(z)= Y_t(z)-X_t(z).$$
Since $\mathbb{P}(C_0 = k) =  v_k(0)$ we have $Z_0(\cdot)=0$. We seek to show that $Z_t$ is identically zero for all $t \ge 0$ by integrating along characteristic curves.

\subsection{Properties of the environmental generating function}

Let us fix some notation for partial derivatives. {\highlight
Given a function $F(\cdot,\cdot)$ or $F_\cdot(\cdot)$ of two variables, where the first or subscripted variable is a time coordinate and the other variable is spatial, i.e. the variable of a generating function, we will write $\dot{f}$ for the partial derivative of $f$ with respect to the time coordinate and $f'$ for the partial derivative with respect to the spatial coordinate. In the case of functions of two variables that are both time coordinates we will not use the dot notation but will write the derivatives explicitly.
}

The generating function analysis in \citet{RT2009} uses the modified moment generating function
\[ V(t,x) = -1 + \sum_{k=1}^\infty v_k(t) e^{-kx}\,,\] 
which was shown to be a solution to the \emph{critical Burgers control problem}
\begin{equation}\label{E: Burgers}\dot{V}(t,x) = -V'(t,x)V(t,x) + e^{-x}\varphi(t)\end{equation} subject to the boundary conditions $V(t,0) = 0$ and $V(0,x) = V_0(x)$. The function $\varphi$ is known as the \emph{control function}. Recall that $\varphi$ appeared in the statement of Theorem \ref{TRthm}; $\varphi(t)$ is the infinitesimal rate at time t at which mass burns and returns to state $1$. The moment generating function $V(t,x)$ is related to our probability generating function $X_t(z)$ by
\begin{equation}
V(t,x) = -1 + X_t(e^{-x}),\label{E: genfuncs}
\end{equation}
and thus \eqref{E: Burgers} is equivalent to
\begin{equation}\label{E: X equation} \dot{X}_t(z) = z X'_t(z) (X_t(z) - 1) + z \varphi(t).\end{equation}

Using equation \eqref{E: genfuncs}, equations (126) and (127) of \cite{RT2009} translate into estimates about the singularity of the probability generating function $X_t(z)$ at $z=1$. In particular for any $w_0\in(0,1]$ and $\overline{t} > T_{gel}$, uniformly on $(t,w) \in  [T_{gel},\overline{t}] \times (w_0, 1]$ we have
\begin{equation}\label{E: sing est} 1 - X_t(1-w^2) = \sqrt{2\varphi(t)}\,w(1 + O(w)) \end{equation}
and
\begin{equation}\label{E: sing est 2} X_t'(1-w^2) = \sqrt{\frac{\varphi(t)}{2}}\,w^{-1}(1 + O(w)). \end{equation}
In fact these estimates hold on $[T_{gel},\overline{t}] \times (0,1]$ as a consequence of the following elementary lemma.

\begin{lemma1}\label{L: power series}
 Let $f(z) = \sum_{k=1}^{\infty} a_k z^k$ be a power series where $\sum_{k} |a_k| = B < \infty$. Then $f$ defines a continuous  function on $\overline{\mathbb{D}}$ with an analytic restriction to $\mathbb{D}$. For all $z \in \mathbb{D}$ we have 
\[|f(z)| \le \sum_{k} |a_k| |z|  \le B|z| \]
 and for each $n \ge 1$, 
\[ |f^{(n)}(z)| \le \sum_{k=1}^\infty B \left|\frac{d^n}{dz^n}z^k\right| \;=\; \frac{n!\,B}{(1-|z|)^{n+1}}\,.\]
In particular each derivative of a probability generating function is locally bounded on $\mathbb{D}$ and the bound does not depend on the probability distribution.
\end{lemma1}

{\highlight
\begin{lemma1}\label{L: X is C1}
 $X_t(z)$ is continuous as a function of $(t,z) \in [0,\infty) \times \overline{\mathbb{D}}$. Moreover, $X_t(z)$ is continuously differentiable as a function of $(t,z) \in \left([0,T_{gel}) \cup (T_{gel},\infty)\right) \times \mathbb{D}$.
\end{lemma1}
\begin{proof}
 According to Theorem~\ref{TRode}, $v_k(t)$ is a continuous function of $t$ for each $k \ge 1$. Since $v_k(t) \ge 0$ and $\sum_{k=1}^\infty v_k(t) =1$ for each $t$, Dini's theorem implies that $t \mapsto (v_k(t))_{k=1}^\infty$ is continuous as a map from $[0,\infty]$ to $\ell^1$. By Lemma~\ref{L: power series}, $|X_t(z) - X_{s}(z)| \le \|v(t) - v(s)\|_1$ for all $z \in \overline{\mathbb{D}}$, so $X_t(z)$ is continuous in $t$, uniformly in $z$, and continuous in $z$ for each $t$. It follows that $X_t(z)$ is jointly continuous as required. 
 
Lemma~\ref{L: power series} shows that $X_t'(z)$ is continuous in $z \in \mathbb{D}$, uniformly in $t$. For each fixed $z \in \mathbb{D}$, the power series $X_t'(z) = \sum_{k=1}^\infty k v_k(t) z^k$ is a uniform limit of continuous functions of $t$, therefore continuous in $t$ for each fixed $z \in \mathbb{D}$. Hence $X_t'(z)$ is continuous on $[0,\infty) \times \mathbb{D}$.
 
To see that $\dot{X}_t(z)$ is continuous on $\left([0,T_{gel}) \cup (T_{gel},\infty)\right) \times \mathbb{D}$, consider the right-hand side of equation \eqref{E: X equation}. Since $\varphi$ vanishes on $[0,T_{gel})$ and is continuous on $[T_{gel},\infty)$, both summands are jointly continuous in $t$ and $z$ in the given domain. Since both partial derivatives are continuous, we conclude that $X_t(z)$ is continuously differentiable on the same domain. 

 We remark that although $\dot{X}_t(z)$ has a jump at $t = T_{gel}$ when $z  \neq 0$, it has left and right limits that depend continuously on $z$.
\end{proof}
}

\subsection{Characteristic curves for the Smoluchowski equations}
\label{Tgelsec}

Recall that the Erd\H{o}s-R\'enyi case $\lambda_n = 0$ is described in the limit by the Smoluchowski coagulation equations with multiplicative kernel:
\begin{equation} \dot{s}_k(t) = -ks_k(t) + \sum_{l=1}^{k-1} l s_l(t)s_{k-l}(t)\quad \text{for $k \ge 1$.} \label{E: Smoluchowski}\end{equation} Even with general initial conditions these can be solved inductively, starting with $k = 1$, so it is easy to see that they have a unique solution. As a warm-up for the analysis later in this section, we describe how the solution can be given using generating functions and characteristic curves. The results in this section are not new, but it is useful to have them in our own terminology.

The infinite system of ODEs \eqref{E: Smoluchowski} is equivalent to the PDE 
\begin{equation}\label{E: Smoluchowski gen} \dot{S}_t(z) = z S'_t(z) ( S_t(z) - 1)\,,\end{equation} for the probability generating function $S_t(z) = \sum_{k=1}^\infty z^k s_k(t)$. Note that \eqref{E: Smoluchowski gen} is precisely \eqref{E: X equation} without control term $z\varphi(t)$. For any $0 \le w < 1$, define $\psi_w(t)$ for $t \ge 0 $ by
\[ \psi_w(t) = w e^{t(1 - S_0(w))}.\]
Then $\frac{d}{dt}{\psi}_w(t) = (1-S_0(w)) \psi_w(t)$ and while $\psi_w(t) < 1$ we have
\begin{align*}
\frac{d}{dt}S_t(\psi_w(t))&=\dot{S}_t(\psi_w(t)) + \frac{d}{dt}{\psi}_w(t)S'_t(\psi_w(t))\\ 
&=\psi_w(t) S'_t(\psi_w(t)) (S_t(\psi_w(t)) - S_0(w)) 
\end{align*}
Gr\"{o}nwall's inequality shows that the unique solution of the above equation is $S_t(\psi_w(t)) = S_0(\psi_w(0))$, so in fact $\psi_w$ is a characteristic curve of \eqref{E: Smoluchowski gen} and satisfies
$$\frac{d}{dt}{\psi}_w(t) = (1-S_t(\psi_w(t))) \psi_w(t).$$
Hence, to find $S_t(z)$ for some $t > 0$ and $z \in (0,1)$, we must find a value of $w$ for which $\psi_w(t) = z$, i.e. 
$$\log w  + t(1-S_0(w))  = \log z.$$ 
The left-hand side of the above equation is a concave function of $w$, its limit as $w \searrow 0$ is $-\infty$ and its limit as $w \nearrow 1$ is $0$, so there is a unique choice of $w \in (0,1)$ such that $\psi_w(t) = z$.  We write $S_0'(1-)$ for $\lim_{z \nearrow 1} S_0'(z)$. The mapping $w \mapsto \psi_w(t)$ is continuous and strictly increasing in $w$ as long as {\highlight $t < 1/wS_0'(w)$}. In particular it is a homeomorphism of $(0,1)$ onto itself for $t \le 1/S_0'(1-)$. But the characteristic curve $\psi_w(t)$ reaches $1$ when $t = \log w / (1 - S_0(w))$.  This means that $S_t(1) = 1$ for $t \le 1/S_0'(1-)$, but $S_t(1) < 1$ for $t > 1/S_0'(1-)$. Thus 
$T_{gel} = 1/S_0'(1-)$
is the \emph{gelation time}: Up to $T_{gel}$ the solution of \eqref{E: Smoluchowski} is \emph{conservative}, meaning that $\sum_{k=1}^\infty s_k(t) = 1$, but after $T_{gel}$ we have $\sum_{k=1}^\infty s_k(t) < 1$, indicating that mass has been lost into the giant component (which is also sometimes referred to as the gel).

\begin{lemma1}\label{L: E(t)}
For all $0 \le t < T_{gel}$, the limiting mean cluster size is given by 
\[ \sum_{k=1}^\infty k v_k(t) = (T_{gel} - t)^{-1}.\]
\end{lemma1}
\begin{proof}
It was remarked above that if $S_0 = V_0$ then the solution of equations \eqref{odecrit} and \eqref{odegeq2} coincides with the solution of the Smoluchowski coagulation equations \eqref{EReqs} up to the time $T_{gel} = \left(\sum_{k=1}^\infty v_k(0)\right)^{-1}$ and no later.

Define 
$$x_k(t) = \sum_{\ell=k+1}^\infty v_\ell(t),\hspace{1pc}E(t)=\sum_{k=1}^\infty k v_k(t)=\sum_{k=0}^\infty x_k(t).$$ Since $\sum_l l^3v_l(0)<\infty$, we have $E(0)<\infty$. Using \eqref{odegeq2} and \eqref{odecrit} we find that $x_0(t) = 1$ and for all $k \ge 1$ 
$$ \dot{x}_k(t) = \sum_{\ell=1}^{k} \ell v_\ell(t) x_{k-\ell}(t) - \varphi(t)\,.$$
It follows that for $t < T_{gel}$, when $\varphi(t) =0$, we have $\dot{x}_k(t) \ge 0$ for all $k \ge 0$, so $E(t)$ is increasing. A convergent series of increasing functions may be differentiated term-by-term and doing so shows that
$\frac{d E(t)}{dt} = E(t)^2.$
Given the initial condition $E(0) = \sum_{k} kv_k(0) = T_{gel}^{-1}$, this has the unique solution 
$E(t) = (T_{gel} - t)^{-1}.$
\end{proof}


\subsection{Characteristic curves for the critical forest fire equations}

We now move on to define characteristic curves for equation \eqref{E: X equation}.

\begin{lemma1}\label{charcurves}
For each $y>T_{gel}$ there exists a unique continuous function $\psi_y:[0,\infty)\to (0,1]$ such that $\psi_y(t)=1$ for all $t \ge y$, $\psi_y(t)<1$ for all $t < y$, and  
\begin{equation} \label{charcurveeq}
\frac{d\psi_y(t)}{dt}=\psi_y(t)\big(1-X_t(\psi_y(t))\big).
\end{equation}
{\highlight
$\psi_y(t)$ is increasing and continuously differentiable on $(0,\infty)$. The function $y \mapsto \psi_y(0)$ is continuous and strictly decreasing, mapping $(T_{gel},\infty)$ onto $(\gamma,1)$ for some $\gamma \in [0,1)$.
}
\end{lemma1}
{\highlight
\begin{remark1}\label{charcurvesconstructionrem}
We will construct the solution to \eqref{charcurveeq} by working backwards from time $y$ to time $T_{gel}$ and then from $T_{gel}$ to $0$. It is convenient to deal separately with the time intervals $[0,T_{gel}]$ and $[T_{gel}, y]$ because $X_t(\cdot)$ has an algebraic singularity at $1$, described by \eqref{E: sing est} and \eqref{E: sing est 2}, when $t\geq T_{gel}$, while $X_t(\cdot)$ has no singularity at $1$ when $t < T_{gel}$. The result of this is that distinct characteristic curves can coalesce at the value $1$, but coalescence occurs only after time $T_{gel}$, since the initial value problem given by \eqref{charcurveeq} with initial condition $\psi_y(0)$ has a unique solution up to time $y$.    
\end{remark1}
}
\begin{remark1}\label{charcurvesrem}
 The utility of the characteristic curve $\psi_y$ lies in the fact that
\begin{align}\label{E: point of charcurves} 
\frac{d}{dt}\left(X_t\left(\psi_y(t)\right)\right) 
&= X_t'(\psi_y(t))\frac{d\psi_y(t)}{dt} + \dot{X}_t\left(\psi_y(t)\right) \notag \\ 
&= \psi_y(t) \big( X_t'(\psi_y(t))(1- X_t(\psi_y(t)) + X_t'(\psi_y(t))(X_t(\psi_y(t)) - 1) + \varphi(t)\big) \notag \\ 
&= \psi_y(t)\varphi(t).
\end{align}
In particular, on $[0,T_{gel})$ where $\varphi \equiv 0$ we see that $X_t(\psi_y(t))$ is constant, so 
$$ X_t(\psi_y(t)) = X_0(\psi_y(0)) $$
and 
$ \frac{d}{dt}\psi_y(t) = \psi_y(t) \left(1 - X_0\left(\psi_y(0)\right)\right)$,
which implies that $ \psi_y(t) = \psi_y(0) e^{t\left(1 - X_0\left(\psi_y(0)\right)\right)}$.
\end{remark1}
\begin{remark1}\label{gammarem}
 We will show in Lemma~\ref{L: curves fill} that $\gamma = 0$, which is to say that $[0,\infty) \times (0,1)$ is filled by characteristic curves, but the proof will rely on Proposition~\ref{Ctdist}.
\end{remark1}
\begin{proof}[Of Lemma \ref{charcurves}]
Let $y>T_{gel}$. {\highlight The characteristic curve $\psi_y(t)$ is defined for $t \ge y$ by $\psi_y(t) = 1$. Our first task is to extend this solution continuously to $[T_{gel}, \infty)$ so that $\psi_y(t) < 1$ for $T_{gel} \le t < y$. To do this we will make a change of variable to remove the singularity, and apply Picard's theorem.  

We aim to express $\psi_y(t)$ in the form $\psi_y(t)=1-\upsilon_y(t)^2$, where $\upsilon_y(t): [T_{gel},y] \to [0,1)$ is continuous and strictly decreasing, satisfies $\upsilon_y(y) = 0$, and for $t \in (T_{gel},y)$ satisfies
\[ \frac{d}{dt}\upsilon_y(t) = \frac{1}{2}(\upsilon_y(t)^2-1)\,\frac{(1-X_t(1-\upsilon_y(t))^2)}{\upsilon_y(t)}\,. \] The point of this change of variable is to enable us to avoid constructing the constant solution $\psi_y(\cdot)=1$ of \eqref{charcurveeq}, which would correspond to $\upsilon_y(\cdot)=0$. Note that this solution would violate our condition that $\psi_y(t)<1$ for $t<y$. To see how the change of variable helps, note that} from equation \eqref{E: sing est} we have
\[ \frac{1- X_t(1-w^2)}{w} = \sqrt{2\varphi(t)} (1 + \mc{O}(w))\]
uniformly on $[T_{gel},y] \times (0,1]$. Hence, the function $F:[T_{gel},y]\times \R\to\R$ given by
$$F(t,w) = 
\begin{cases} 
1 & \text{for }w>1\\
\frac{1-X_t(1-w^2)}{w} & \text{for }w\in(0,1]\\ 
\sqrt{2\varphi(t)} & \text{for }w\leq 0
\end{cases}$$
is continuous and strictly positive on $[T_{gel},y]\times\R$. In fact, using \eqref{E: sing est 2} as well as \eqref{E: sing est} we have that uniformly on $(t,w) \in [T_{gel},y] \times (0,1]$,
\begin{align*}
\frac{\partial{F}(t,w)}{\partial{w}} 
&= 2X_t'(1-w^2) - w^{-2}(1 -X_t(1-w^2))\\
&= 2\sqrt{\frac{\varphi(t)}{2}} w^{-1} \big(1 + \mc{O}(w)\big) - \sqrt{2 \varphi(t)} w^{-1}\big(1+\mc{O}(w)\big)\\
&= \mc{O}(1)
\end{align*}
It follows that, within $[T_{gel},y]\times\R$, the function $F(t,w)$ is continuous with respect to $t$ and Lipschitz with respect to $w$, uniformly in $t$. Therefore, it follows from Picard's Theorem that there is a unique solution $\upsilon_y:[T_{gel},y]\to\R$ of the equation
\begin{equation}\label{upsilondef}
\upsilon_y(t) = \int_t^y \frac{1}{2}(1-\upsilon_y(s)^2)F(s,\upsilon_y(s))\,ds,\hspace{2pc}\upsilon_y(y)=0
\end{equation}
and, further, that $\upsilon_y(\cdot)$ is continuous.

{\highlight
We aim now to show that $0 < \upsilon_y(t) < 1$ for all $t \in [T_{gel},y)$, and that $\upsilon_y$ is strictly decreasing there. From \eqref{upsilondef} we obtain \begin{equation}\label{upsilonode}\frac{d}{dt}\upsilon_y(t) = \frac{1}{2}(\upsilon_y(t)^2-1)F(t,\upsilon_y(t))\,,\end{equation} and hence $\upsilon_y$ is also the (unique) solution of the equation
\begin{equation}\label{upsilonaltdef} f(t) = \upsilon_y(t_0) + \int_{t_0}^t \frac{1}{2}(f(s)^2-1)F(t,f(s))\,ds\,, \end{equation} for any $t_0 \in [T_{gel},y]$. 

Since $(w^2-1)F(t,w)\geq 0$ when $w\geq 1$ it follows that $\upsilon_y(t)<1$ for all $t \in [T_{gel},y]$, for if this were to fail at $t=t_0$ then~\eqref{upsilonaltdef} shows that $\upsilon_y$ would be increasing on $[t_0,y]$, contradicting $\upsilon_y(y)=0$. 

 We now show that $\upsilon_y(t) > 0$ for $t \in [T_{gel},y)$.  Suppose that $\upsilon_y(t) \le 0$ for some $t < y$. Then consider \[t_0 = \sup\{t \in [T_{gel},y)\set \upsilon_y(t) \le 0\}\,.\] We must have $t_0 < y$, since $\frac{d}{dt}\upsilon_y(t)\mid_{t=y} < 0$. Since $\upsilon_y$ is continuously differentiable by~\ref{upsilonode}, we must have $\upsilon_y(t_0) = 0$, and $\frac{d}{dt}\upsilon_y(t_0) \ge 0$ which contradicts~\eqref{upsilonode}.  Hence $0 < \upsilon_y(t) < 1$ for $T_{gel} \le t<y$, the integrand on the right-hand side of \eqref{upsilondef} is strictly positive, and $t\mapsto \upsilon_y(t)$ is strictly decreasing.
}

Since $0 < \upsilon_y(t) < 1$ for all $t \in [T_{gel},y)$, from the definition of $F$ we have
$$ \frac{d\upsilon_y(t)}{dt} = \frac{1}{2}(\upsilon_y(t)^2 - 1)\left(\frac{1 - X_t(1 - \upsilon_y(t)^2)}{\upsilon_y(t)}\right)\,.$$ 
We now define
$$\psi_y(t)=1-\upsilon_y(t)^2$$
and it follows that $\psi_y$ has the desired properties: it is continuous and strictly increasing on $[T_{gel},y]$ and satisfies \eqref{charcurveeq} there, $\psi_y(y)=1$, and $0 < \psi_y(t)<1$ for $t\in[T_{gel},y)$.  Note that $\psi_y(t) \to 1$ as $t \nearrow y$ and 
$$\frac{d}{dt}\psi_y(t) = -2\upsilon_y(t) \frac{d}{dt}\upsilon_y(t) \to 0 \quad \text{as $t \nearrow y$.} $$

Having constructed $\psi_y(t)$ on the interval $[T_{gel},y]$, we can extend the solution uniquely back from time $T_{gel}$ to time $t=0$, using Picard's theorem applied directly to equation \eqref{charcurveeq}. We use $\Psi=\psi_y(T_{gel})$, which has already been defined above, as our `initial' condition.

To this end, for $t\in[0,T_{gel}]$ we define
$$
G(t,z)=
\begin{cases}
\psi_y(T_{gel})(1-X_{t}(\psi_y(T_{gel}))) & \text{for }z> \psi_y(T_{gel})\\
z(1-X_t(z)) & \text{for }z\in[0,\psi_y(T_{gel})]\\
0 &\text{for }z<0.
\end{cases}
$$
Then $G(t,z)$ is continuous on $[0,T_{gel}]\times\R$ and for $z\in (0,\psi_y(T_{gel})$ and $t \in [0,T_{gel}]$ we have
\begin{align*}
\frac{\p G(t,z)}{\p z}=1-X_t(z)-zX'_t(z).
\end{align*}
Since $\psi_y(T_{gel})<1$, Lemma \ref{L: power series} implies that $X$ and $X'$ are uniformly bounded in $[0,T_{gel}]\times[0,\psi_y(T_{gel})]$. Hence $\frac{\p G(t,z)}{\p z}=\mc{O}(1)$, so $G(t,z)$ is Lipschitz in $z$, uniformly in $t$. Thus, by Picard's Theorem the equation
\begin{equation*}
\frac{d\psi_y(t)}{dt}=\psi_y(t)G(t,\psi_y(t)),\hspace{2pc}\psi_y(T_{gel})=\Psi
\end{equation*}
has a unique solution over $[0,T_{gel}]$. Since $G(t,z)\geq 0$ it follows that $\psi_y$ is increasing and since $G(t,z)=0$ for $z\leq 0$ it follows that $\psi_y(0)>0$. This completes the construction of the characteristic curves.

{\highlight
By construction $\psi_y$ is continuously differentiable on $(0,T_{gel})$, on $(T_{gel},y)$ and on $(y, \infty)$. Since it has matching left and right one-sided derivatives at $T_{gel}$ and at $y$, $\psi_y$ is continuously differentiable on $(0,\infty)$ as required. 

To prove the claims about the function $y \mapsto \psi_y(0)$, note first that $\frac{d}{dt}\psi_y(t) \le \psi_y(t)$ and hence for any $t < y$ we have $1 \ge \psi_y(t) \le \psi_y(0) e^t$. Letting $t \to y$ we find $\psi_y(0) \ge e^{-y}$, so $0 < \psi_y(0) < 1$ as claimed. The function $y \mapsto \psi_y(0)$ is strictly decreasing since otherwise we would have $y < y'$ such that $\psi_y(0) \ge \psi_{y'}(0)$, in which case there would be some maximal $s < y$ such that $\psi_y(s) = \psi_{y'}(s)$, at which there would be more than one solution to the initial value problem
\[\dot{f}(t) = f(t) \left(1 - X_t(f(t))\right),\quad f(s) = \psi_y(s)\,,\] contrary to Picard's Theorem.

We now show that $y \mapsto \psi_y(0)$ is a decreasing bijection by exhibiting its inverse. For each $z \in (0,1)$ there exists a unique solution $f_z$ of the initial value problem \[\dot{f}(t) = f(t) \left(1 - X_t(f(t))\right),\quad f(0) = z\,\] taking values in $(0,1)$, on some maximal domain $[0,y(z))$, since the right-hand side of this problem is locally Lipschitz. On the interval $[0, \min(y(z),T_{gel}))$ the same calculations as in Remark~\ref{charcurvesrem} show that $X_t(f_z(t))$ is constant in $t$ and hence $f_z(t) = z e^{t(1 - X_0(z))}$. Since $(1-X_t(f_z(t))$ is non-negative, the solution is increasing on $[0,y(z))$, so either $f_z(t) \nearrow 1$ as $t \nearrow y(z)$ or $y(z) = \infty$. In the former case the explicit solution on $[0, T_{gel}]$ shows that $y > T_{gel}$, and by the uniqueness proved above we must have $f_z = \psi_{y(z)}$ on $[0,y(z))$. In the latter case we would have $z < \psi_y(0)$ for every $y \in [T_{gel},\infty)$. Hence if we define $\gamma = \inf\{ z \in (0,1)\,:\, y(z) < \infty\}$, then we have exhibited an inverse mapping for $y \mapsto \psi_y(0)$, defined on the interval $(z,1)$.
} 
\end{proof}

\begin{remark1}
The method of Lemma \ref{charcurves} can also be used to construct the characteristic curves {\highlight $\xi_s(t)$ defined in equation (66) in Section 3.2 of \cite{RT2009}, using the relationship $\xi_s(t) = -\log \psi_s(t)$.} 
\end{remark1}

\begin{defn}\label{Cexpl}
We say that $C_t$ explodes at (the random) time $t\geq 0$ if $C_{t}=1$ and for some (random) $\epsilon>0$ $C_s\neq 1$ for all $s\in(t-\epsilon,t)$.
\end{defn}
{\highlight
 An equivalent definition is that $C_t$ explodes at time $t$ if and only if $C_t$ makes infinitely many jumps in $(s,t)$ for every $s < t$. In particular the event that $t$ is an explosion time and the number of explosions that occur in $[0,t]$ are both measurable with respect to $\mathcal{G}_{t-}$, where $\{G_t\}_{t \ge 0}$ is the filtration generated by the {\cadlag} process $C_t$.
}

\begin{lemma1}\label{L: explosion prob} For any $y > T_{gel}$ and $0 \le s < y$, we have
\begin{equation}\label{E: charcurve meaning}
Y_s(\psi_y(s)) = \mathbb{P}[C\text{ does not explode in }[s,y]\,]
\end{equation}
and $\mathbb{P}[C\text{ explodes at time }y] = 0$. 
{\highlight
Furthermore,
\[\mathbb{P}[C\text{ explodes during }[0,T_{gel}]\,] = 0\,.\]
}
\end{lemma1}
\begin{proof}
Fix a time $s \ge 0$ and define let $\tau_s = \inf\{ t > s \,:\, \text{ $C$ explodes at time $t$}\}$. Note that $\tau_s$ is a previsible stopping time.  Then the following process is defined for $u \in [s,y)$:
$$M_y(u) = 
\begin{cases} 
\psi_y(u)^{C_u} & \text{if }u < \tau_s, \\ 
0 &\text{if }u \ge \tau_s. 
\end{cases}
$$
{\highlight
In particular $M_y(s) = \psi_y(s)^{C_s}$, and $M_y$ is adapted to $\{\mathcal{G}_{t}\}$. We claim that $M_y(\cdot)$ is a $\mathcal{G}_t$-martingale.
 By conditioning on the first jump in $(t, t+\Delta)$ being of size $j$, we obtain 
\begin{eqnarray*}\mathbb{P}[\text{$C$ jumps at least twice in $(t,t+\Delta)$}]  & = & \int_{0}^\Delta ke^{-sk} \sum_{j=1}^\infty v_j(t+s) \left(1 - e^{-(\Delta-s)(k+j)}\right)\,ds\,\\
& = & k e^{-k\Delta} \int_0^\Delta 1 - X_{t+s}\left(e^{-(\Delta - s)}\right)\,ds \\
 & \le & k \Delta \sup_{0 \le s \le \Delta} \left(1 - X_{t+s}\left(e^{-\Delta}\right)\right)\,.
\end{eqnarray*}
By Dini's theorem, $X_t(e^{-\Delta})$ converges locally uniformly to $1$ as $\Delta \searrow 0$, so the last expression above is $o(\Delta)$.
}
It follows that conditional on $C_u = k$ and $u < \tau_s$ the drift of $M_y$ at time $u$ is 
\begin{align*} 
&k \psi_y(u)^{k-1} \frac{d}{du}\psi_y(u) + k \sum_{l=1}^\infty (\psi_y(u)^{k+l} - \psi_y(u)^k)v_l(u)\\ 
&\hspace{4pc}= k \psi_y(u)^{k-1} \left( \frac{d}{du}\psi_y(u) + \psi_y(u) \left( X_u(\psi_y(u)) - 1  \right)\right) \\ 
&\hspace{4pc}= 0.
\end{align*}
That the final line above is $0$ follows from Lemma \ref{charcurves}. Since $M_y(u)$ is bounded, by the martingale convergence theorem we may extend it to a martingale $M_y(u)$ defined for $u \in [s,y]$ that is a.s.~continuous at $u=y$.

Note that $C_u\nearrow\infty$ as $u\nearrow \tau_s$. Hence, by Lemma \ref{charcurves} if $\tau_s < y$ then $M_y(u) \to 0$ as $u \nearrow \tau_s$ and $M_y(u) = 0$ for $u \in [\tau_s, y)$. If $\tau_s > y$ then $M_y(u) \to 1$ as $u \nearrow y$. If $\tau_s = y$ then all we know is that $M_y(y) \in [0,1]$. Therefore
\[\mathbf{1}(\tau_s > y) \le M_y(y) \le \mathbf{1}(\tau_s \ge y)\,.\] Taking conditional expectations on $\mathcal{G}_u$,
$$ \mathbb{P}[\tau_s > y \,|\, \mathcal{G}_u] \le M_y(u) \le   \mathbb{P}[\tau_s \ge y \,|\, \mathcal{G}_u].$$
Hence, for any $y' > y$, we have
$$ \psi_{y'}(s)^{C_s}  \le  \mathbb{P}[\tau_s \ge y' \,|\, \mathcal{G}_s] \le  \mathbb{P}[\tau_s > y \,|\, \mathcal{G}_s]   \le  \psi_y(s)^{C_s}.$$
Taking expectations of the above equation, we obtain
$$ Y_s\left(\psi_{y'}(s)\right) \le \mathbb{P}\left[\tau_s \ge y'\right] \le \mathbb{P}\left[\tau_s > y\right] \le Y_s\left(\psi_y(s)\right).$$
 By Lemma~\ref{charcurves} we can choose $y'$ so as to make $\psi_{y'}(s)$ as close as we like to $\psi_y(s)$, and it follows that $\mathbb{P}[\tau_s = y] = 0$. We obtain also that
\begin{equation}\label{psi and tau} Y_s\left(\psi_y(s)\right) = \mathbb{P}[\text{$C$ does not explode in $(s,y)$}]  =   \mathbb{P}[C\text{ does not explode in }(s,y]\,]\,.\end{equation}
{\highlight
Finally, to show that $C$ almost surely does not explode in $[0,T_{gel}]$, by Lemma~\ref{charcurves} we have $\psi_y(0) \nearrow 1$ as $y \searrow T_{gel}$, and $\lim_{z \nearrow 1}Y_0(z) = 1$, so
\begin{eqnarray*} \mathbb{P}[C\text{ does not explode during }[0,T_{gel}]\,] & = & \lim_{y \searrow T_{gel}}\mathbb{P}[C\text{ does not explode during }[0,y]\,] \\ & = & \lim_{y \searrow T_{gel}} Y_0(\psi_y(0)) \;=\; 1\,.
\end{eqnarray*}
}
\end{proof}

\begin{remark1} \label{psiexplode}
By Lemma \ref{L: explosion prob}, if $\mathbb{P}(C_s = 1) > 0$, which always holds for $s > T_{gel}$, then from \eqref{psi and tau} we obtain 
\begin{align}\label{charexpl}
\psi_y(s) 
&= \mathbb{P}[\text{$C$ does not explode in $(s,y)$} \,|\, C_s = 1] \notag\\
&= \mathbb{P}[\text{$C$ does not explode in $(s,y]$}\,|\,C_s = 1].
\end{align}
If $v_1(0) = 0$ then $\mathbb{P}(C_s = 1) =  0$ for $s \le T_{gel}$, in which case conditioning on $C_s = 1$ does not make sense. In this case we can consider a modified version $\hat{C}_t$ of $C_t$ that is started in state $1$ at time $s$, and the same argument shows that \eqref{charexpl} holds with $\hat{C}$ in place of $C$.
\end{remark1}

\begin{corollary1}\label{Ypsitto1}
Let $y > T_{gel}$. Then $Y_s\left(\psi_y(s)\right) \nearrow 1$ as $s \nearrow y$. 
\end{corollary1}
{\highlight
\begin{proof} 
 This follows from Lemma~\ref{L: explosion prob} on applying the dominated convergence theorem to the indicator functions
 $\mathbf{1}(\text{$C$ does not explode in $[s,y]$})$ as $s \nearrow y$.
\end{proof}
}

\subsection{Evolution of the watched cluster distribution}\label{evolutionsec}

We now seek an analogue of equation~\eqref{v1eq} for $u_1$. For $t \in [0,\infty)$ define 
$$\Phi(t)=\E\left[\#\{s\in[0,t)\, : \,C\text{ explodes at time }s\}\right]\,.$$
Recall that $C$ spends an exponential time of mean $1$ in state $1$ after each explosion. Thus we can stochastically bound the number of explosions in $[0,t]$ by $1$ plus a Poisson process of rate $1$,  hence $\Phi(t) \le 1 + t$ for all $t$. Lemma~\ref{L: explosion prob} implies that $\Phi$ is continuous, but we have not yet shown that $\Phi$ is differentiable, so we cannot write down a differential equation for $u_1$. For this reason, it is convenient instead to use integral equations to describe the evolution of $(u_k)$.  By examining the transitions of $C$ we obtain 
\begin{align*}
u_1(t)&=v_1(0)-\int_0^tu_1(s)\,ds + \Phi(t),\\ 
u_k(t)&=v_k(0)-\int_0^tku_k(s)\,ds + \int_0^t\sum\limits_{l=1}^{k-1}lu_l(s)v_{k-l}(s)\,ds\,\quad \text{for }k \ge 2.
\end{align*} 
 Note that the appearance of $v_{k-l}$ corresponds to the fact that each time $C$ jumps it increases by a sample of $l\mapsto v_l(t)$. From the above two equations, for $|z| < 1$ we obtain
\begin{equation}Y_t(z)=X_0(z)-\int_0^t zY'_s(z)(1-X_s(z))\,ds + z\Phi(t)\,.\phantom{aaaaa}\label{Yeq}
\end{equation} 
Similarly, from \eqref{v1eq} and Theorem \ref{TRode} we can show that 
\begin{equation} \label{Xeq}
  X_s(z)=X_0(z)+\int_0^tzX'_s(z)(1-X_s(z))\,ds + z\int_0^t \varphi(s)\,ds.
 \end{equation}  
Combining \eqref{Yeq} and \eqref{Xeq} and using the initial condition $Z_0 = Y_0 - X_0 = 0$ we obtain
\begin{equation} Z_t(z)=zI(t)-\int_0^tzZ'_s(z)(1-X_s(z))\,ds
\label{Zeq}
\end{equation}
where $I$ is the continuous function defined by
$$I(t)=\Phi(t)-\int_0^t\varphi(s)\,ds.$$
Since the integrand in \eqref{Zeq} is bounded (by Lemma~\ref{L: power series}), we see that $Z_t(z)$ is continuous in $t$ for each fixed $z \in \mathbb{D}$. Differentiating \eqref{Zeq} under the integral we find that for $|z| < 1$ we have
\begin{equation} Z'_t(z) = I(t)-\int_0^t \frac{d}{dz}\left[zZ'_s(z)(1-X_s(z))\right]\,ds\,.
\label{Zdasheq}
\end{equation}
To justify this by showing that the integral is absolutely convergent, expand the derivative and apply Lemma~\ref{L: power series} to bound the result in terms of $|z|$, independently of $t$. This also shows that for each fixed $z$ with $|z| < 1$, $Z'_t(z)$ is a continuous function of $t$, and then Lemma~\ref{L: power series} implies that $Z_t'(s)$ is jointly continuous on $[0,\infty) \times \mathbb{D}$.

For $t\in[0,\infty)$ and $|z|<1$ define
\begin{equation}\label{Rdef}
R_t(z)=Z_t(z)-zI(t)=-\int_0^t zZ'_s(z)(1-X_s(z))\,ds.
\end{equation}
We are aiming to show that both $Z$ and $I$ are identically zero. For each fixed $z$ with $|z| < 1$, we see from the integral expression in \eqref{Rdef} that $R_t(z)$ is differentiable with respect to both $t$ and $z$, satisfying
\begin{eqnarray}\label{Rder}
\dot{R}_t(z) & = & -zZ'_t(z)(1-X_t(z))\,, \\
\label{Rder2} R_t'(z) & = & Z_t'(z) - I(t)\,.\end{eqnarray}
Hence $R'$ is continuous on $[0,\infty) \times \mathbb{D}$. Using \eqref{Rder} with Lemma~\ref{L: X is C1} we find also that $\dot{R}_t(z)$ is jointly continuous in $t$ and $z$ and hence $R_t(z)$ is continuously differentiable on $[0,\infty) \times \mathbb{D}$.
 
\begin{lemma1}\label{Rpsitto1}
Let $y>T_{gel}$. Then $R_y(\psi_y(t)) \to -I(y)$ as $t \nearrow y$. \end{lemma1}
\begin{proof}
By definition, $R_y(z) = Y_y(z) - X_y(z) - zI(y)$. By Lemma~\ref{L: X is C1} and Lemma~\ref{L: explosion prob} we know that $X_t(z)$ and $zI(t)$ are both jointly continuous in $(t,z) \in [0,\infty) \times [0,1]$. From Lemma \ref{charcurves} we have $\psi_y(t)\nearrow 1$ as $t\nearrow y$, so $X_t(\psi_y(t)) \to X_y(1) = 1$ as $t \nearrow y$. Corollary \ref{Ypsitto1} gives $Y_t(\psi_y(t)) \to 1$ as $t \nearrow y$, so
$$R_y(\psi_y(t))  = Y_y(\psi_y(t))-X_t(\psi_y(t))-\psi_y(t)I(t) \to 1 - 1 -I(y)$$ as $t \nearrow y$, as required. 
\end{proof}

We are now in a position to prove Proposition \ref{Ctdist}. Recall that Proposition \ref{Ctdist} stated that for all $t\in[0,\infty)$ and all $l\in\N$, $\Prob\left[C_t=l\right]=v_l(t)$.

\begin{proof}[Proof of Proposition \ref{Ctdist}.]
Let $y>T_{gel}$. By Lemma \ref{charcurves} we have $\psi_y(t)\in[0,1)$ for $t<y$. Let 
$$\eta_y(t)=R_t(\psi_y(t))$$
for all $y>T_{gel}$ and $t\in[0,y)$. 

Combining the continuity and continuous differentiability of $\psi_y(\cdot)$ proved in Lemma \ref{charcurves} with the properties of $R_t(z)$ proved above, we find that $t\mapsto \eta_y(t)$ is continuous on $[0,y)$ and continuously differentiable on $(0,y)$. Using \eqref{Rder}, \eqref{Rder2} and \eqref{charcurveeq} we compute, for $0 < t < y$,
\begin{align}
\frac{d\eta_y(t)}{dt}&=\dot{R}_t(\psi_y(t)) + \frac{d\psi_y(t)}{dt}R'_t(\psi_y(t)))\notag\\
&=-\psi_y(t)Z'_t(\psi_y(t))(1-X_t(\psi_y(t)))+\psi_y(t)Z'_t(\psi_y(t))(1-X_t(\psi_y(t)))-\frac{d\psi_y(t)}{dt}I(t)\notag\\
&= - \psi_y(t)(1-X_t(\psi_y(t))I(t).\label{etader}
\end{align}
Hence for all $t \in [0,y)$ we have 
$$\eta_y(t)-\eta_y(0)= - \int_0^t\psi_y(s)(1-X_s(\psi_y(s)))I(s)ds.$$
Using \eqref{Rdef} and Lemma \ref{Rpsitto1} we have 
$\eta_y(0)=R_0(\psi_y(0))=0$ and $\lim_{t\nearrow y}\eta_y(t) = -I(y).$
Hence
$$I(y)=\int_0^y\psi_y(s)(1-X_s(\psi_y(s)))I(s)ds$$
which implies that
$$|I(y)|\leq\int_0^y|I(s)|ds.$$

The above equation holds for all $y>T_{gel}$. By Lemma \ref{L: explosion prob} and \eqref{v1eq} we have $I(y)=0$ for all $y<T_{gel}$. Using Gr\"{o}nwall's inequality, this shows that $I$ is identically zero. Hence from \eqref{etader} we have $\frac{d\eta_y(t)}{dt}=0$ and since $\eta_y(0)=0$ we have $\eta_y=0$. Hence, from \eqref{Rdef}, for all $y>T_{gel}$ and $t\in[0,y)$ we have $Z_t(z)=zI(t)=0$ for all $z\in[\psi_y(t),1)$. By Lemma \ref{charcurves} and the identity theorem this implies that $Z_t$ is identically zero for each $t < y$. Since $y > T_{gel}$ was arbitrary, this shows that for every $t \in [0,\infty)$ we have $Y_t = X_t$ and hence $u_k(t)=v_k(t)$ for all $k \in \N$. 
\end{proof}

{\highlight
\subsection{$C_t$ almost surely explodes infinitely often}\label{explodesiosec}

We are now in a position to establish an important property of $C$, namely that it explodes infinitely often.

\begin{lemma1}\label{L: average burn rate} For all $t \ge 0$, 
\begin{equation}
\mathbb{E}(1/C_t) = \mathbb{E}(1/C_0) + \int_{0}^t \left(\varphi(s) - 1/2\right) \, ds\,.\label{phiconvweak}
\end{equation}
\end{lemma1}
\begin{proof} By Proposition~\ref{Ctdist} and Fubini's theorem, \[\mathbb{E}(1/C_t) = \sum_{k=1}^\infty \frac{v_k(t)}{k} = \int_0^1 \frac{X_t(z)}{z} \,dz\,.\] 
Hence, using \eqref{E: X equation} and Fubini's theorem again,
\begin{eqnarray*}\mathbb{E}(1/C_t) - \mathbb{E}(1/C_0) & = & \int_0^1 \frac{X_t(z) - X_0(z)}{z}\,dz \; =\;  \int_0^1 \int_0^t \frac{\dot{X}_t(z)}{z} \,dt\,dz \\
 & = & \int_0^1 \int_0^t X_t'(z)\left(X_t(z) - 1\right) + \varphi(t)\,dt\,dz\\
 & = & \int_0^t \int_0^1 \frac{d}{dz}\left(\frac{X_t(z)^2}{2} - X_t(z)\right) \,dz + \varphi(t) \, dt\,,
\end{eqnarray*}
This gives the desired result since $X_t(1) = 1$ and $X_t(0) = 0$ for all $t$.
\end{proof}
\begin{lemma1}\label{L: curves fill}
 For every $t \ge 0$, $\lim_{y \to \infty} \psi_y(t) = 0$. The characteristic curves $\psi_\cdot(\cdot)$ fill $[0,\infty) \times (0,1)$ and 
$C_t$ almost surely explodes infinitely often.
\end{lemma1}
\begin{proof}
 For any $y > T_{gel}$ and $0 \le t \le s < y$ we have
 \[ 1 \,\ge\, X_s\left(\psi_y(s)\right) \,=\, X_t\left(\psi_y(t)\right) + \int_{t}^s \psi_y(u)\varphi(u) \,du\, \ge \, \psi_y(t)\int_t^s \varphi(u)\,du\,.\]
 By Lemma~\ref{L: average burn rate} we have 
 \[ \int_t^s \varphi(u) \,du = \frac{s-t}{2} + \mathbb{E}\left(1/C_s\right) - \mathbb{E}\left(1/C_t\right) \ge \frac{s-t}{2} - 1\,,\]
 so letting $s \nearrow y$ we obtain
 \[0 \le \psi_y(t) \le \frac{1}{\frac{y-t}{2} - 1}\,.\]
Hence for every $t \ge 0$ and $0 < z < 1$ we can find $y$ large enough to ensure $\psi_y(t) < z$. Now it follows by the arguments used to prove Lemma~\ref{charcurves} that there exists $y > T_{gel}$ such that $\psi_y(t) = z$. 
 
 To conclude that $C_t$ almost surely explodes infinitely often, we use Lemma~\ref{L: explosion prob}, to see that
 \[\mathbb{P}(\text{$C$ does not explode after time $t$}) = \lim_{y \to \infty} Y_t(\psi_y(t)) = Y_t(0) = 0\,\]
as required.
\end{proof}

\begin{remark1}
Note that, since $\frac{1}{C_t}\in(0,1]$, \eqref{phiconvweak} establishes a weak sense in which $\varphi(t)$ approaches $\frac{1}{2}$ as $t\to\infty$. We conjecture that, in fact, $\varphi(t)\to\frac{1}{2}$ as $t\to\infty$.
\end{remark1}
}

\section{Coupling}
\label{couplingsec}

In this section we will prove Theorem \ref{CntoC} by coupling the pair $(C^n,C)$ so that, for any fixed $T>0$, if we take $n$ sufficiently large then, with probability close to $1$ we have $C^n(\cdot)=C(\cdot)$ for all but a small proportion of the time interval $[0,T]$ and at the exceptional times $C^n(\cdot)$ and $C(\cdot)$ are nevertheless close in the compact state space $E$. The coupling divides into two parts:
\begin{enumerate}
\item If $C$ and $C^n$ are small, and equal, in size then their respective jump rates are exponential random variables of similar rate and we can use Theorem \ref{TRthmupg} to couple the size of the correspond jumps. Note that this incurs a small probability of failure, caused by the jump rates not quite being equal and by $v^n$ and $v$ (which define the jump distributions) being not quite equal.
\item Eventually $C$ and $C^n$ become large enough that probably of failure incurred above is too high to control. At this point we rely on conservativity \eqref{odecrit}, which combined with Proposition \ref{Ctdist} and Theorem \ref{TRthmupg} implies that a cluster which is already large will burn quickly (in both $C$ and $C^n$). Once they burn, with high probability they stay in state $1$ for long enough to enable recoupling and we can once again use Theorem \ref{TRthmupg}.
\end{enumerate}
We now proceed with the argument, which is rather involved. Our first step is to show that large clusters burn quickly and once that is done we will construct the coupling.

\begin{lemma1}\label{bigdiefast1}
Let $\epsilon>0$. Let $T>0$ with and $\delta\in(0,T)$. Then there exists $K,N\in\N$ such that for all $n\geq N$ and all $k\geq K$,
$$\Prob\left[\exists t\in [\delta,T]\text{ such that }\inf\limits_{s\in [t-\delta,t]}C^n_s>k\right]<\epsilon.$$
\end{lemma1}
\begin{proof}
Let $T,\epsilon>0$, $\delta\in(0,T)$ and $A_{k,\delta}=\left\{\exists t\in [\delta,T]\text{ such that }\inf_{s\in [t-\delta,t]}C^n_s>k\right\}.$
On the event $A_{k,\delta}$ we have $\int_0^T \1\{C^n_s>k\}ds\geq\delta$ and thus
\begin{equation}\label{bigdiefasteq1}
\delta\,\Prob\left[A_{k,\delta}\right]\leq \E\left[\int_0^T\1\{C^n_s>k\}ds\right].
\end{equation}

We now seek an upper bound on the right hand side of \eqref{bigdiefast1}. Since our tagged vertex $p$ was sampled uniformly from $[n]$ we have $P\left[C^n_s=j\right]=\E\left[v^n_j(s)\right]$. Hence, 
\begin{align}
\E\left[\int_0^T\1\{C^n_s<k\}\,ds\}\right]&=T-\E\left[\int_0^T\1\{C^n_s\leq k\}\,ds\right]\notag\\
&=T-\sum\limits_{j=1}^k\E\left[\int_0^T v^n_j(s) \,ds\right]\label{bigdiefasteq2}.
\end{align}
By Theorem \ref{TRthmupg} and the fact that $v^n_j(t)\in[0,1]$, for each $j\in\N$ we have
\begin{equation}\label{vnconvmoll}
\E\left[\int_0^T v^n_j(s) \,ds\right]\to\int_0^T v_j(s)\,ds\, \quad \text{as $n\to\infty$}\,.
\end{equation}
Since $\sum_{j=1}^\infty v_j(s)=1$ we have 
$T=\int_0^T\sum_{j=1}^\infty v_k(s)\,ds = \sum_{j=1}^\infty \int_0^T v_k(s)\,ds.$
Hence we can choose $K\in\N$ such that for all $k\geq K$,
\begin{equation}\label{bigdiefasteq3}
T-\sum\limits_{j=1}^k\int_0^T v_j(s) \,ds <\frac{\epsilon\delta}{2}.
\end{equation}
Using \eqref{vnconvmoll}, we choose $N\in\N$ such that for all $j=1,\ldots,K$ and $n\geq N$, we have
$$\left|\E\left[\int_0^T v^n_j(s)\,ds\right]-\int_0^T v_j(s)\,ds\right|<\frac{\epsilon\delta}{2K}.$$
Putting the above equation and \eqref{bigdiefasteq3} into \eqref{bigdiefasteq2} we get
$$\E\l[\int_0^T\1\{C^n_s<k\}\,ds\r]< \epsilon\delta.$$ 
Combining this equation with \eqref{bigdiefast1} obtains the stated result.
\end{proof}

\begin{lemma1}\label{bigdiefast2}
Let $\epsilon>0$. Let $T>0$ and $\delta\in(0,T)$. Then there exists $K\in\N$ such that for all $k\geq K$,
$$\Prob\left[\exists t\in [\delta,T]\text{ such that }\inf\limits_{s\in [t-\delta,t]}C_s>k\right]<\epsilon.$$
\end{lemma1}
\begin{proof}
Let $\epsilon,T>0$ and $\delta\in[0,T]$. By Proposition \ref{Ctdist} we have $\Prob\left[C_s=j\right]=v_j$ for all $s$. We set
$A_{k,\delta}=\left\{\exists t\in [\delta,T]\text{ such that }\inf_{s\in [t-\delta,t]}C_s>k\right\}$
and then, in similar style to \eqref{bigdiefasteq1} and \eqref{bigdiefasteq2}, we obtain 
\begin{align*}
\delta\,\Prob\left[A_{k,\delta}\right]&\leq \E\left[\int_0^T\1\{C_s>k\}\,ds\right]=T-\sum\limits_{j=1}^k\int_0^T v_j(s)\,ds.
\end{align*}
By \eqref{odecrit} we have
$T=\sum_{j=1}^\infty\int_0^T v_j(s)ds$
so we can choose $K\in\N$ such that for all $k\geq K$,
$$T-\sum_{j=1}^k\int_0^Tv_j(s)ds<\epsilon\delta$$
and we are done.
\end{proof}

Our next step is to construct a coupling between $C$ and $C^n$ (where $n$ is still to be chosen). In fact we will define a {\cadlag} $E$ valued process $\wt{C}$ which has the same distribution as $C$ and is coupled to $C^n$. Our coupling will be such that with probability close to $1$ the distance $d_E(\wt{C}_t,C^n_t)$ remains small up until a given large time $T$. We will also define a process $S$, taking values in $\{\vec{0},\vec{1}\}$. The process $S$ acts as a \textit{state bit}; for as long as we can keep $\wt{C}$ and $C^n$ close we have $S=\vec{0}$. The time 
$$\tau=\inf\{s>0\set S_s=\vec{1}\}$$
records the time at which our coupling fails to keep $\wt{C}$ and $C^n$ close. 

{\highlight
At this point we make a minor modification to our model $\mc{Z}^n$:
\begin{itemize}
\item[$(\dagger)$] Growth clocks corresponding to (unordered) pairs $(i,j)$ where $i,j\in\mc{C}^n$ and $i\neq j$ ring at twice their normal rate (i.e.~at rate $\frac{2}{n}$ instead of $\frac{1}{n}$).
\end{itemize}
This modification has no effect on the clusters which form in $\mc{Z}^n$, since any such pair $(i,j)$ were already part of the same connected cluster. We may assume $(\dagger)$ without any change to the dynamics of the partition of vertices into clusters. We are not interested here in the internal structure of the clusters, so we assume $(\dagger)$ from now on. Thanks to $(\dagger)$, we obtain the following properties: 
\begin{itemize}
\item[(A)] Given that $C^n_t=k$, the time until one of the growth clocks of some $(i,j)$ with $i\in\mc{C}^n_t$ next rings is exponential with rate $k$.
\item[(B)] When the corresponding edge appears, at time $t'$, the effect is that a vertex $i$ is sampled uniformly from $\mc{C}^n_{t'-}$ and a second vertex $j$ is sampled uniformly from $[n]$ (and an edge is created between these two vertices). 
\end{itemize}

\begin{remark1}\label{R: exact growth rate}
Without $(\dagger)$, given $C^n_t=k$, the time until one of the growth clocks of some $(i,j)$ with $i\in\mc{C}^n_t$ rings would be exponential with rate 
\begin{equation}\label{Rnk}
R^n_k=\frac{1}{n}\l(k(n-k)+k+\binom{k}{2}\r).
\end{equation}
This does not match the jump rate of $\wt{C}^n$, causing additional error terms that we would then need to control (using that fact that $R^n_k\to k$ as $n\to\infty$). Further, without $(\dagger)$ we would not have property (B) and this would create yet more error terms.
\end{remark1}

Let $H(i,j)$ denote the growth clock for the (unordered) pair $(i,j)$. Let $I_{0,0}=0$ and define inductively
\begin{align*}
I_{a+1,0}&=\inf\{s>I_{a,0}\set \mc{C}^n\text{ is burned at time }s\}\\
I_{a,b+1}&=\inf\{s>I_{a,b}\set s<I_{a+1,0}\text{ and }\exists i\in \mc{C}^n_{s-}, j\in\N \text{ such that }H(i,j)\text{ rings at }s\}.
\end{align*}
Thus, for $a>0$, $I_{a,0}$ is the $a^{th}$ time at which $\mc{C}^n$ burns and $I_{a,b}$ is the $b^{th}$ time after its $a^{th}$ burn at which a growth clock of some $(i,j)$ with at least one of $i$ and $j$ currently in $\mc{C}^n$ rings. Note that the (random) set $\mathscr{I}=\{I_{a,b}\set a,b\geq 0, I_{a,b}\text{ is defined}\}$ is well ordered in time with respect to the lexicographic integer ordering of the indexes $(a,b)$. With slight abuse of language, we say that $I_{a,b}\in\mathscr{I}$ is a fire if $b=0$ and growth if $b>0$.

Consider the time $I_{a,b}$, for $b\geq 1$, corresponding to, say, the growth clock $H(i_{a,b},j_{a,b})$ where by property (B) we have that (conditionally on $\mc{Z}^n$ prior to $I_{a,b}$, and independently of all else) $i_{a,b}$ is sampled uniformly from $\mc{C}^n_{I_{a,b}-}$ and $j_{a,b}$ is sampled uniformly from $[n]$. It is advantageous to introduce some extra notation with which to carry out the sampling of $j_{a,b}$. We recall that $\mc{Z}^n$ is invariant under permutations of the labels $[n]$ of the vertices, so (for convenience) at all times we label the vertices in increasing order of cluster size. To sample $j_{a,b}$, we sample a uniform random variable $U_{a,b}$ on $[0,1]$, independently of all else, and we set 
$$j_{a,b}=\min\l\{j'\in[n]\set U_{a,b}<\frac{j'}{n}\r\}.$$ 
The value of 
$$L_{a,b}=C^n_{I_{a,b}-}(j_{a,b})$$ 
is then taken to be
\begin{equation}\label{paintbox}
\sum\limits_{l=1}^{L_{a,b}}v^n_l(I_{a,b}-)\leq U_{a,b} < \sum\limits_{l=1}^{L_{a,b}+1}v^n_l(I_{a,b}-),
\end{equation}
which simply states that $C^n_{I_{a,b}-}(j_{a,b})$ has conditional distribution $l\mapsto v^n_l(I_{a,b}-)$. We additionally sample a random variable $\wt{L}_{a,b}$ given by
\begin{equation}\label{paintbox2}
\sum\limits_{l=1}^{\wt{L}_{a,b}}v^n_l(I_{a,b}-)\leq U_{a,b} < \sum\limits_{l=1}^{\wt{L}_{a,b}+1}v^n_l(I_{a,b}-).
\end{equation}
Thus $\wt{L}_{a,b}$ has conditional distribution $l\mapsto v_l(I_{a,b})$.

We define
$$\mathscr{I}=\{I_{a,b}\set a,b\geq 0\text{ and }(a,b)\neq(0,0)\}.$$
For any $t\in[0,\infty)$ we define
$$t^\oplus=\min\{s\in\mathscr{T}\set s>t^{\oplus}\}.$$
Note that (almost surely) each time $t^\oplus$ is either a fire or a growth.
}

Let $K\in\N$ be arbitrary, for now; it will be given a fixed value later. We partition $\{\vec{0},\vec{1}\}\times E\times E$ into six subsets, corresponding to six cases in the definition of the joint evolution of $(S_t,C^n_t,\wt{C}_t)$. These are
\begin{align*}
E_1&=\{\vec{0}\}\times\{(k,k)\set k\leq K\}\\
E_2&=\{\vec{0}\}\times\{(k,\overline{k})\set k,\overline{k}> K\}\\
E_3&=\{\vec{0}\}\times\{(k,1)\set k> K\}\\
E_4&=\{\vec{0}\}\times\{(1,\overline{k})\set \overline{k}> K\}\\
E_5&=\{\vec{1}\}\times E\times E\\
E_6&=(\{0,1\}\times E\times E)\setminus\cup_{m=1}^5E_m.
\end{align*}

{\highlight
In similar style to \eqref{paintbox}, at time $0$ we sample a uniform random variable $U$ on $[0,1]$, take our watched point $p$ to be $p=\min\{p'\in[n]\set U<\frac{p'}{n}\}$ and take the size $C^n_0$ of our watched cluster to be
\begin{equation}\label{Cninit}
\sum\limits_{l=1}^{C^n_0}v^n_l(0)\leq U < \sum\limits_{l=1}^{C^n_0+1}v^n_l(0).
\end{equation}
Similarly, we define the initial state of $\wt{C}$ by
\begin{equation}\label{Cinit}
\sum\limits_{l=1}^{\wt{C}_0}v_l(0)\leq U < \sum\limits_{l=1}^{\wt{C}_0+1}v_l(t).
\end{equation}
}
If $C^n_0=\wt{C}_0$ then we set $S_0=\vec{0}$, otherwise we set $S_0=\vec{1}$. The evolution of $(S_t,C^n_t,\wt{C}_t)$ then proceeds as follows; the infinitesimal evolution is different depending on which $E_i$ the process $(S,C^n,\wt{C})$ is in.
\begin{itemize}
\item If $(S_t,C^n_t,\wt{C}_t)=(\vec{0},k,k)\in E_1$ then, at time $t^{\oplus}$,
\begin{itemize}
\item if $t^\oplus=I_{a,b}$ is a growth, and both $j_{a,b}\notin \mc{C}^n_{t^{\oplus}}$ and $L_{a,b}=\wt{L}_{a,b}$, it jumps to $(\vec{0},C^n_t+L_{a,b},\wt{C}_t+\wt{L}_{a,b})$.
\item if $t^\oplus=I_{a,b}$ is a growth, and both $j_{a,b}\notin \mc{C}^n_{t^{\oplus}}$ and $L_{a,b}\neq\wt{L}_{a,b}$, it jumps to $(\vec{1},C^n_t+L_{a,b},\wt{C}_t+\wt{L}_{a,b})$.
\item if $t^\oplus=I_{a,b}$ is a growth, and $j_{a,b}\in \mc{C}^n_{t^{\oplus}}$, it jumps to $(\vec{1},C^n_t,\wt{C}_t+\wt{L}_{a,b})$.
\item if $t^\oplus$ is a fire, it jumps to $(\vec{1},1,\wt{C}_t)$. 
\end{itemize}

\item While $(S,C^n,\wt{C})=(\vec{0},k,\overline{k})\in E_2$, the processes $C^n$ and $\wt{C}^n$ evolve independently of one another. The evolution of $C^n$ is already specified, $S$ remains constant at $\vec{0}$ and the evolution of $\wt{C}$ is that
\begin{itemize}
\item if $\wt{C}_t=k$, then at rate $k$ it jumps to $\wt{C}_t+L$, where $L$ is sampled from $v_l(\alpha)$ and $\alpha$ is the jump time.
\end{itemize}
Note that this could result in infinitely many jumps of $\wt{C}$ inside $E_2$, in which case, by definition of $E$, at the time of the accumulation point of these jumps $\wt{C}$ enters state $1$.

\item While $(S_t,C^n_t,\wt{C}_t)=(\vec{0},k,1)\in E_3$ then, 
\begin{itemize}
\item At rate $1$, it jumps to $(\vec{1},C^n_t,\wt{C}_t+L)$,  where $L$ is sampled from $v_l(\alpha)$ and $\alpha$ is the jump time.
\end{itemize}
If $t^\oplus$ occurs before this (potential) jump then, at time $t^{\oplus}$,
\begin{itemize}
\item if $t^\oplus=I_{a,b}$ is a growth, it jumps to $(\vec{0},C^n_t+L_{a,b}\1\{j_{a,b}\notin \mc{C}^n_{t^\oplus}\},1)$.
\item if $t^\oplus$ is a fire, it jumps to $(\vec{0},1,1)$.
\end{itemize}

\item While $(S_t,C^n_t,\wt{C}_t)=(\vec{0},1,\overline{k})\in E_4$, the processes $C^n$ and $\wt{C}^n$ evolve independently of one another. The evolution of $\wt{C}$ is that
\begin{itemize}
\item if $\wt{C}_t=k$, then at rate $k$ it jumps to $\wt{C}_t+L$, where $L$ is sampled from $v_l(\alpha)$ and $\alpha$ is the jump time. In this case $S$ and $C^n$ remain constant.
\end{itemize}
If $t^\oplus$ occurs before this (potential) jump then, at time $t^\oplus$,
\begin{itemize}
\item if $t^\oplus=I_{a,b}$ is a growth, it jumps to $(\vec{1},1+L_{a,b}\1\{j_{a,b}\notin \mc{C}^n_{t^\oplus}\},\wt{C}_t)$.
\item if $t^\oplus$ is a fire, there is no change.
\end{itemize}

\item While $(S,C^n,\wt{C})\in E_5$, we have $S=\vec{1}$ and the processes $C^n$ and $\wt{C}$ evolve independently of one another. The evolution of $C^n$ is already specified and the evolution of $\wt{C}$ is the same as defined above in the case of $E_2$.

\item The process $(S,C^n,\wt{C})$ does not enter $E_6$. 
\end{itemize}

\begin{remark1}
For as long as $S=\vec{0}$, the transitions between the $E_i$ of $(S,C^n,\wt{C})$ follow the cycle $E_1\to E_2\to (E_3\cup E_4)\to E_1$. Precisely one of $E_3$ and $E_4$ is visited in each such cycle. Once $(S,C^n,\wt{C})$ has entered $E_5$ (which happens as soon as $S=\vec{1}$), it never leaves. 
\end{remark1}

By comparing each case in turn, it can be seen that the evolution specified for $C^n$ matches that given in Section \ref{introsec}. Further, it is clear from the above definition that the {\cadlag} process $(S,C^n,\wt{C})$ is strongly Markov with respect to its generated filtration $(\mc{F}_t)$. Note that this filtration is the product of the filtration of $\mc{Z}^n$ and the additional randomness introduced above (i.e.~$\wt{C}$, $U_{a,b}$, $L_{a,b}$ and so on). With mild abuse of notation, we extend our probability measure $\P$ to be a measure on $\sigma(\cup_{t\in[0,\infty)}\mc{F}_t)$.

\begin{lemma1}\label{CwtC}
The processes $\wt{C}$ and $C$ have the same distribution.
\end{lemma1}
\begin{proof}
When $(S,C^n,\wt{C})\in E_2\cup E_4\cup E_5$, the evolution of $\wt{C}$ defined above is trivially the same as that given in Definition \ref{Cdef}. If $(S,C^n,\wt{C})\in E_1$, then $\wt{C}$ jumps at rate $k$ (corresponding to the next $t^{\oplus}$ that is a growth). On such a jump at time, say $\alpha$, the jump causes a displacement $L$ with distribution (conditional on $\alpha$) given by $l\mapsto v_l(\alpha)$. If $(S,C^n,\wt{C})\in E_3$ then $\wt{C}=1$ and in this case jumps of $\wt{C}$ occur at rate one, with the jump distribution $l\mapsto v_l(\alpha)$. 

Thus, in all cases $\wt{C}$ has the same jump rates and jump distributions as $C$. Since the paths of $C$ and $\wt{C}$ are characterized entirely by their jump times and corresponding displacements, $C$ and $\wt{C}$ have equal distribution.
\end{proof}

\begin{remark1}
The process $\wt{C}$ clearly depends on $n$. However, it follows from Lemma \ref{CwtC} that the distribution of $\wt{C}$ does not depend on $n$. It is for this reason that we choose not to add a superscript $n$ onto $U_{a,b}$, $L_{a,b}$, etc.
\end{remark1}

We now aim to show that, given $T\in(0,\infty)$, $K$ and $n$ can be chosen so that $\P\l[\tau\leq T\r]$ is arbitrarily small. Our first step, which will allow us to make use of \eqref{paintbox}, \eqref{paintbox2}, \eqref{Cninit} and \eqref{Cinit}, is the following lemma.

{\highlight
\begin{lemma1}\label{paintboxbound}
Let $(x_n)_{n=1}^\infty,(x'_n)_{n=1}^\infty\sw[0,1]$ be random sequences such that $\sum_n x_n=\sum x'_n=1$. Let $U\in [0,1]$ be a uniform random variable on $[0,1]$ which is independent of $(x_n)$ and $(x'_n)$. Define $c,c'\in\N$ by
\begin{equation}\label{cdef}
\sum\limits_{k=1}^c x_k\leq U < \sum\limits_{k=1}^{c+1}x'_k,\hspace{2pc}\sum\limits_{k=1}^{c'} x'_k\leq U < \sum\limits_{k=1}^{c'+1}x'_k.
\end{equation}
Suppose that $\eta>0$ is such that
\begin{equation}\label{xconds}
\P\l[\exists k\leq K, |x_k-x'_k|\geq \frac{\eta}{K^2}\r]\leq \frac{\eta}{K}, \hspace{2pc} \P\l[\sum\limits_{k>K}x_k\leq \eta\r]=1.
\end{equation}
Then $\P\l[c=c'\r]\geq 1-6\eta.$
\end{lemma1}
\begin{proof}
We note that
$$\sum_{k>K} x'_k=\sum_{k>K} x_k + \sum_{k=1}^K (x_k-x'_k)$$
so that (in similar style to \eqref{TRupg2}) from \eqref{xconds} we have
\begin{equation}\label{xconds2}
\P\l[\sum_{k>K}x'_k>2\eta\r]\leq \eta.
\end{equation}
Writing $s_k=\sum_{l=1}^k x_k$ and $s'_k=\sum_{l=1}^k x'_k$, from \eqref{cdef} we have that
\begin{equation}\label{cc'}
\{c\neq c'\}\sw\l\{U\geq \sum\limits_{k=1}^K x_k\r\}\cup \l\{U\geq \sum\limits_{k=1}^K x'_k\r\}\cup \l(\bigcup\limits_{k=1}^K\l\{U\in\l[\min\l(s_k,s'_k\r),\max\l(s_k,s'_k\r)\r]\r\}\r).
\end{equation}
Using that
$$|\max(s_k,s'_k)-\min(s_k,s'_k)|\leq \sum\limits_{l=1}^k|x_k-x'_k|\leq \sum\limits_{l=1}^K|x_k-x'_k|$$
in \eqref{cc'}, along with \eqref{xconds} and \eqref{xconds2}, we obtain 
\begin{align*}
\P\l[c\neq c'\r]\leq \eta + 3\eta + \frac{\eta}{K} + \sum\limits_{k=1}^K\frac{\eta}{K}\leq 6\eta
\end{align*}
as required.
\end{proof}
}

Let $\epsilon>0$, let $T\in(0,\infty)$. Let $J\in\N$ be large enough that
\begin{equation}\label{Jdef}
\P\l[I_{J+1,0}\leq T\r]<\epsilon.
\end{equation}
Let $\delta>0$ be such that
\begin{equation}\label{deltadef}
\delta J\leq \epsilon.
\end{equation}
Note that it is trivial that such a $J$ exists, since $C^n$ spends an exponential time of rate $1$ at state $1$ upon each return. We now choose the value of $K$. First, using the definition of $E$ and Assumption \ref{crit}, let $K,N\in\N$ be large enough that, for all $n\geq N$,
\begin{align}
d(1,K)&\leq\epsilon,\label{K1}\\
K\lambda_nT&\leq\epsilon,\label{K2}\\
\frac{K^2J}{N}&\leq \epsilon\label{K3}
\end{align}
Then, by Theorem \ref{TRthmupg}, combined with Lemmas \ref{bigdiefast1} and \ref{bigdiefast2} (the latter of which applies by Lemma \ref{CwtC}), we may increase $K$ and $N$ so that for all $n \ge N$,
\begin{align}
\P\l[\exists t\in[\delta,T]\text{ such that }\inf\limits_{s\in[t-\delta,t]}C^n_s>K\r]&\leq\frac{\epsilon}{J},\label{K4}\\
\P\l[\exists t\in[\delta,T]\text{ such that }\inf\limits_{s\in[t-\delta,t]}\wt{C}_s>K\r]&\leq\frac{\epsilon}{J},\label{K5}\\
\P\l[\sup\limits_{l\in\N}\sup\limits_{s\in[0,T]}|v^n_l(t)-v_l(t)|>\frac{\epsilon}{K^3J}\r]&\leq\frac{\epsilon}{K^2J}.\label{K6}
\end{align}
Using Dini's Theorem, we may increase $K$ so that also
\begin{align}
\sup\limits_{s\in[0,T]}\sum\limits_{k>K}v_k(s)<\epsilon.\label{K7}
\end{align} 
Finally we may increase $N$ if necessary to ensure that inequalities \eqref{K1}--\eqref{K7} hold simultaneously.
\begin{lemma1}\label{statebitlemma}
If $S_t=\vec{0}$ then $d_E(C^n_t,\wt{C}_t)\leq\epsilon$.
\end{lemma1}
\begin{proof}
This follows immediately from \eqref{K1}, the definition of $(S,C^n,\wt{C})$ and the definition of $d_E$.
\end{proof}

\begin{lemma1}\label{tauT}
It holds that $\P\l[\tau\leq T\r]\geq 1-22\epsilon$.
\end{lemma1}
\begin{proof}
Using the definition of the tagged vertex $p$, \eqref{Cninit}, \eqref{Cinit}, \eqref{K6} and \eqref{K7}, we can apply Lemma \ref{paintboxbound} to $C^n_0$ and $C$ (with $\eta=\epsilon$, $x_k=v_k(0)$ and $x'_k=v_k^n(0)$) and obtain that 
\begin{equation}\label{B1}
\P\l[C^n_0\neq\wt{C}_0\r]\leq 6\epsilon.
\end{equation}
In view of the above equation and Lemma \ref{statebitlemma}, in order to prove the current lemma we must control the probabilities of $S$ exiting state $\vec{0}$ during time $[0,T]$. This exit can occur in several different ways, as can be seen from the definition of $(S,C^n,\wt{C})$. We go through each possible case (in the same order as they occur within the definition of $(S,C^n,\wt{C}$)) and establish a bound on the probability of each. Let us first examine the transitions out of $E_1$ that can lead to $S=\vec{1}$.
\begin{itemize}
\item
When $C^n$ makes a jump from a state in $\{1,\ldots,K\}$ at time $I_{a,b}$, the probability that this jump has $L_{a,b}\neq \wt{L}_{a,b}$ can be bounded above using Lemma \ref{paintboxbound}. Using, \eqref{paintbox}, \eqref{paintbox2}, \eqref{K6} and \eqref{K7} (with $\eta=\frac{\epsilon}{KJ}$), we obtain that this probability is bounded above by $\frac{6\epsilon}{JK}$. Now, by \eqref{Jdef}, with probability at least $1-\epsilon$ the process $C^n$ exits state $1$ at most $J$ times during $[0,T]$. On this event, there can be at most $JK$ jumps of $C^n$ from states in $\{1,\ldots,K\}$ during $[0,T]$. Hence, 
\begin{equation}\label{B2}
\P\l[\exists\,a,b\text{ such that }I_{a,b}\leq T\text{ and }C_{I_{a,b}-} \le K\text{ and }L_{a,b}\neq \wt{L}_{a,b}\r]\leq \epsilon+KJ\frac{6\epsilon}{JK}=7\epsilon.
\end{equation}

\item
When $C^n$ makes a jump from a state in $\{1,\ldots,K\}$ at time $I_{a,b}$, the probability that this jump has $j_{a,b}\in \mc{C}^n_{I_{a,b}}$ is equal to the probability that a uniform random element of $\{1,\ldots n\}$ is within $\{1,\ldots K\}$, which is itself $\frac{K}{n}$. As in the above case, by \eqref{Jdef}, with probability at least $1-\epsilon$ there are at most $JK$ such jumps. Hence, by \eqref{K3}
\begin{equation}\label{B3}
\P\l[\exists\,a,b\text{ such that }I_{a,b}\leq T\text{ and }j_{a,b}\in\mc{C}^n_{I_{a,b}}\r]\leq\epsilon+JK\frac{K}{N}\leq 2\epsilon.
\end{equation}

\item 
Given that $C^n=k$, the rate at which $\mc{C}^n$ burns is $k\lambda_n$. Hence, the probability that $\mc{C}^n$ sees a fire at some time $t\in[0,T]$ for which $C^n_t\leq K$ is bounded above by $1-e^{K\lambda_n T}$. Hence, by \eqref{K2} we have that
\begin{equation}\label{B4}
\P\l[\exists t\in[0,T]\text{ such that }t^\oplus\text{ is a fire and }C^n_{t}\leq K\r]\leq \epsilon.
\end{equation}

\end{itemize}
There are no transitions out of $E_2$ that can lead to $S_t=\vec{1}$. We now move on to the transitions out of $E_3\cup E_4$ that can lead to $S=\vec{1}$. 
\begin{itemize}
\item
The only possible transition out of $E_3$ which leads to $S=\vec{1}$ is if $\wt{C}$ makes a jump. This occurs at rate $1$. By \eqref{K4} combined with \eqref{Jdef}, with probability $1-2\epsilon$, the total time spent in $E_3$ is at most $J\delta$. Thus, using also \eqref{deltadef} we have
\begin{equation}\label{B6}
\P\l[\exists t\in[0,T]\text{ such that }(S_t,C^n_t,\wt{C}_t)\in E_3\text{ and }\wt{C}\text{ jumps during }[t,t^\oplus]\r]\leq 2\epsilon +(1-e^{-J\delta})\leq 3\epsilon.
\end{equation}

\item
The only possible transition out of $E_4$ which leads to $S=\vec{1}$ is if $C^n$ makes a jump. This is essentially the same case as $E_3$ (see above), except that the roles of $C^n$ and $\wt{C}$ are reversed. We apply the same argument as is used above, with \eqref{K5} replacing \ref{K4}, to obtain
\begin{equation}\label{B7}
\P\l[\exists t\in[0,T]\text{ such that }(S_t,C^n_t,\wt{C}_t)\in E_4\text{ and }C^n\text{ jumps before }\wt{C}\r]\leq 3\epsilon.
\end{equation}
\end{itemize}

Since $S=\vec{1}$ within $E_5$, and $E_6$ is never visited, there are no more jumps in which the value of $S$ can change from $\vec{0}$ to $\vec{1}$. Summing up our error terms in \eqref{B1}-\eqref{B7}, we obtain the required result.
\end{proof}

\begin{remark1}
It is clear from the proof that of Lemma \ref{tauT} that, during $[0,\tau)$, the process $(S,C^n,\wt{C})$ spends most of its time in $E_1$, during which $C^n=\wt{C}$.
\end{remark1}

\begin{proof}[Of Theorem \ref{CntoC}.]
By Lemma \ref{CwtC}, the $E\times E$ valued process $(C^n,\wt{C})$ is a coupling of $C^n$ and $C$. By Lemmas \ref{statebitlemma} and \ref{tauT} we have
$\Prob[\sup_{t\leq T}d_E(\wt{C},C^n)>\epsilon] \leq \Prob\left[\tau\geq T\right] \leq 22\epsilon$ 
for all $n\geq N$.
\end{proof}

\begin{ack}
The authors 
thank Bal\'{a}zs R\'{a}th and an anonymous
 referee for providing a number of corrections and suggestions to improve the presentation
 and shorten some of the proofs. BT acknowledges support from the Hungarian Scientific Research Fund (OTKA, grants no.~100473 and 109684), and the Leverhulme Trust (International Network grant ``Laplacians, Random Walks, Quantum Spin Systems'').
\end{ack}

\bibliographystyle{plainnat}

\end{document}